\documentclass[preprint,1p,times]{elsarticle}

\usepackage{amsmath,amssymb,amsthm,amsfonts,amscd} 
\usepackage{enumerate}
\usepackage[font=footnotesize, labelfont={sf,bf}, margin=1cm]{caption}
\usepackage{graphicx}

\newtheorem{lemma}{Lemma}
\newtheorem*{prop1A}{Proposition 1$_a$}
\newtheorem*{prop2A}{Proposition 2$_a$}
\newtheorem*{prop1B}{Proposition 1$_b$}
\newtheorem*{prop2B}{Proposition 2$_b$}
\newtheorem*{prop1C}{Proposition 1$_c$}
\newtheorem*{prop2C}{Proposition 2$_c$}
\newtheorem*{prop3}{Proposition 3}
\newtheorem*{prop4}{Proposition 4}
\newtheorem*{prop5}{Proposition 5}
\newtheorem*{prop6}{Proposition 6}
\newtheorem*{prop7}{Proposition 7}
\newtheorem*{theorem}{Theorem}

\newtheorem*{claim}{Claim}
\theoremstyle{plain}
\newtheorem*{thm1A}{Theorem 1$_a$}
\newtheorem*{thm2A}{Theorem 2$_a$}
\newtheorem*{thm1B}{Theorem 1$_b$}
\newtheorem*{thm2B}{Theorem 2$_b$}
\newtheorem*{thm1C}{Theorem 1$_c$}
\newtheorem*{thm2C}{Theorem 2$_c$}

\theoremstyle{definition}
\newtheorem{theo}{Theorem}
\newtheorem{definition}[theo]{Definition}

\journal{Annales de l'Iinstitut Henri Poincar{e}}

\begin{document}

\begin{frontmatter}

\title{Gluing scalar-flat manifolds with vanishing mean curvature on the boundary} 
 
\author{Demetre Kazaras}

\ead{demetre@uoregon.edu}

\address{Department of Mathematics, University of Oregon, Eugene, OR 97403-1222 USA}

\begin{abstract}
We establish a gluing theorem for solutions of a Yamabe
problem for manifolds with boundary studied by J. Escobar in the
mid 90's.  We begin with two compact Riemannian manifolds
with boundary, each scalar-flat, of vanishing boundary mean curvature,
and equipped with a common submanifold $K$.  Under suitable geometric 
conditions, we produce a 1-parameter family of metrics on 
the generalized
connect sum along $K$, each of
vanishing scalar curvature and constant boundary mean curvature.
Assuming an extra non-degeneracy hypothesis, we can arrange for these 
metrics to have vanishing boundary mean curvature.
Moreover, these metrics converge to the original metrics 
away from the gluing site in the $\mathcal{C}^2$ topology.
\end{abstract}


\end{frontmatter}

\section{Introduction}
  Given a closed $n$-dimensional manifold $M$ and a conformal class $C$, 
  the classical
  Yamabe problem asks if there is a metric in $C$ of constant
  scalar curvature. 
  Such metrics are critical points of the Einstein-Hilbert functional
  \[
  C\to\mathbb{R},\quad 
  g\mapsto\frac{\frac{n-2}{4(n-1)}\int_M R_gd\mu_g}{\mathrm{Vol}_g(M)^{\frac{n-2}{n}}}
  \]
  restricted to the class $C$. See section 1 for a description of our notation.
  When the solution of this problem \cite{S} was nearly a decade old, J. Escobar
  introduced generalizations to compact manifolds $M$ with non-empty
  boundary $\partial M$. The natural functional to consider in the context of 
  a boundary is 
  the {\em total scalar curvature plus total mean curvature} \cite{Ar}. 
  In order to make this quantity scale-invariant, it must be renormalized.
  In the case of the classical Yamabe problem this is accomplished by dividing
  the total scalar curvature by $\mathrm{Vol}_g(M)^{\frac{n-2}{n}}$. 
  For manifolds with boundary, however, one may choose to renormalize with 
  respect to the volume of the interior, the boundary, or a combination of the two.
    
  In \cite{E1}, 
  Escobar studies the following family of functionals   
  \[
  C\to\mathbb{R},\quad g\mapsto\frac{\frac{n-2}{4(n-1)}\int_MR_gd\mu_g+\frac{n-2}{2(n-1)}\int_{\partial M}H_gd\sigma_g}{a\mathrm{Vol}_g(M)^{\frac{n-2}{n}}+(1-a)\mathrm{Vol}_g(\partial M)^{\frac{n-2}{n-1}}}
  \]
  where $a$ is a fixed number in the interval $[0,1]$. For any value of $a$, 
  critical points of this functional
  are metrics of constant scalar curvature with constant mean curvature 
  on the boundary.
  For $a=1$, critical points are scalar-flat and for $a=0$ critical points 
  have vanishing mean
  curvature on the boundary. These extremal cases are studied, respectively, 
  in \cite{E3} and \cite{E1} where critical points are found for 
  a large class of $M$ and $C$.  Notice that scalar-flat metrics with vanishing 
  mean curvature on the boundary are critical points of this functional for any 
  value of $a$. Conformal classes which contain such metrics 
  are called {\em Yamabe-null}.

\begin{figure}[htb!]
\begin{center}
\begin{picture}(0,0)
\end{picture}
\includegraphics[height=3in]{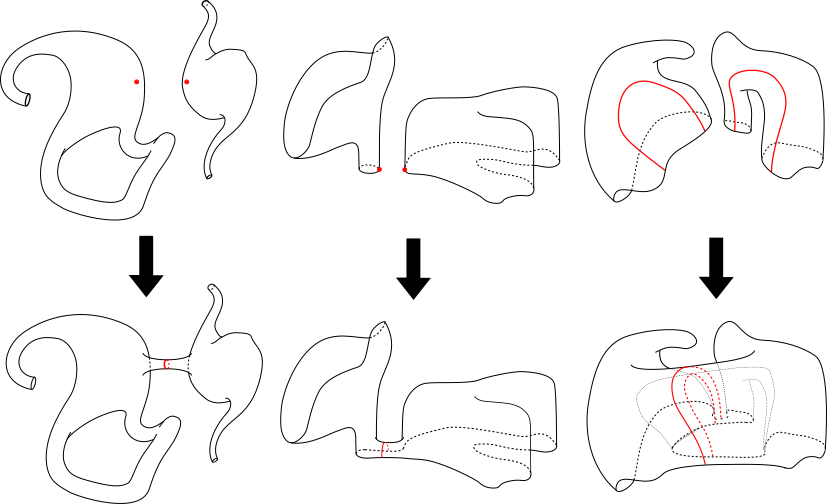}
\end{center}

\caption{Schematic description of the generalized connected sum for 
(from left) an interior, boundary, and relative embedding.}
\label{fig:schematic}
\end{figure}

  In this paper we determine to what extent two Yamabe-null
  manifolds with boundary may be glued along a common submanifold to
  produce a third Yamabe-null manifold.
  Gluing constructions have a rich and storied history in geometric analysis,
  too extensive to satisfactorily survey here.
  For our construction, we will adopt a particular scheme introduced by 
  L. Mazzieri in \cite{M1}
  for gluing closed manifolds with non-zero constant scalar curvature.
  His work generalizes results
  of D. Joyce \cite{J} on connected sums of closed manifolds of non-zero 
  constant scalar
  curvature (see also \cite{Mazzeo}). In \cite{M2} Mazzieri
  considers the more delicate problem of gluing two closed Yamabe-null
  manifolds to produce another manifold of
  vanishing scalar curvature. In general, this process may be
  obstructed if one of the two original manifolds is Ricci-flat.
  In the present paper we encounter a similar
obstruction which, naturally, involves the second fundamental form of
the original manifolds' boundaries -- obstructors to our process can
be identified as Ricci-flat manifolds with totally geodesic boundary.
Our construction is flexible enough to also glue along submanifolds which themselves
have boundary meeting the ambient boundary orthogonally.
This requires a new geometric construction and we naturally encounter
a family of elliptic problems with mixed Dirichlet-Neumann boundary
conditions for which we must provide new a priori estimates.

Let us describe the main result, first in the case where gluing occurs along 
a submanifold embedded away from the boundary which we call an 
{\it interior embedding}.  Let
$(M_1,g_1)$ and $(M_2,g_2)$ be $n$-dimensional compact manifolds which
are scalar-flat and have vanishing boundary mean curvatures.
Moreover, suppose that each is equipped with an isometric embedding of
a closed $k$-dimensional manifold $(K,g_K)$, denoted by $\iota_*:K\to
\mathring{M}_*$ ($*=1,2$).  Assuming that the isometry
$\iota_1\circ\iota_2^{-1}$ extends to an isomorphism of the normal
bundles of $K$, we may form $M:=M_1\#_KM_2$, the generalized connected
sum along $K$ by removing small tubular neighborhoods and using the
bundle isomorphism to identify annular regions (see Figure \ref{fig:schematic}). In sections 2 and 3,
we begin by producing and studying a 1-parameter family of metrics
$g_\varepsilon$ on $M$ transitioning between $g_1$ and $g_2$ on a
neighborhood of the surgery site.  The metrics $g_\varepsilon$ can be
thought of as attaching $M_1$ and $M_2$ by a thin, short $K$-shaped
tube which becomes thinner as $\varepsilon$ decreases.  This family serves as a starting point for an iterative
construction described in sections 4 and 5 which produces a family of
metrics conformal to $g_\varepsilon$, each scalar flat and of constant
boundary mean curvature.  More formally, we prove the following.

\begin{thm1A}
  Let $(M_1,g_1)$, $(M_2,g_2)$ be compact
  $n$-dimensional manifolds with non-empty boundaries. 
  Assume that
  \[
  R_{g_1 }\equiv 0, \ \ H_{g_1}\equiv 0, \ \ 
   R_{g_2 }\equiv 0, \ \ H_{g_2}\equiv 0, \ \ \mbox{and} \ \
\mathrm{Vol}_{g_1}(\partial M_1)= \mathrm{Vol}_{g_2}(\partial
       M_2).
  \]
       Given isometric embeddings $\iota_1:K\to\mathring{M}_1$, $\iota_2:K\to\mathring{M}_2$ of a closed $k$-dimensional manifold 
       $(K,g_K)$ of codimension $m:=n-k\geq3$ with
       isomorphic normal bundles, there exists a family of scalar-flat
       metrics $\{\tilde{g}_\varepsilon\}_{\varepsilon\in(0,\varepsilon_0)}$
       (for some $\varepsilon_0>0$) on $M= M_1\#_KM_2$ with constant
       boundary mean curvature
\[
|H_{\tilde{g}_\varepsilon}|=\mathcal{O}(\varepsilon^{m-2}).
\]
Moreover, for each $\varepsilon$, $\tilde{g}_\varepsilon$ is conformal to $g_{*}$ away
from a fixed tubular neighborhood of 
$\iota_*(K)$ in $M_*$ and
$\tilde{g}_\varepsilon\to g_{*}$
on compact sets of
$M_*\setminus\iota_*(K)$ in the $\mathcal{C}^2$ topology as
$\varepsilon\to0$ for $*=1,2$.
\end{thm1A}
The above codimension restriction allows spheres in fibers of 
the normal bundles to carry curvature, which will be required in our
construction.  If neither of the original manifolds $(M_1,g_1)$, $(M_2,g_2)$ are 
Ricci-flat with vanishing second
fundamental form of the boundary, more can be accomplished -- we may alter 
this construction
in an $\varepsilon$-small non-conformal manner, so that the resulting
metrics have vanishing boundary mean curvature.
\begin{thm2A}
  Assume, in addition to the conditions in Theorem 1$_a$, that both
  manifolds $(M_1,g_1)$ and $(M_2,g_2)$ are not Ricci-flat with
  vanishing second fundamental form of their boundaries.  Then
  there exists a second family of scalar-flat metrics
  $\{\hat{g}_{\varepsilon}\}_{\varepsilon\in(0,\varepsilon_0)}$ on $M=M_1\#_KM_2$ with
  vanishing boundary mean curvature.  Moreover, $\hat{g}_\varepsilon\to
  g_{*}$ on compact sets of $M_*\setminus\iota_*(K)$ in the
  $\mathcal{C}^2$ topology as $\varepsilon\to0$ for $*=1,2$. 
\end{thm2A}
As mentioned earlier, we additionally consider gluing along boundaries 
i.e. when the
embedding of $K$ has a non-trivial intersection with
$\partial M_1$ and $\partial M_2$.  Carrying out the construction in this case requires
substantial changes and new estimates which are contained in sections
2 and 3.  It is convenient to break into two further cases: that in which $K$ is closed
and embedded into the boundaries $\partial M_*$ and that in
which $K$ itself has a boundary $\partial K$ with $\mathring{K}$ and
$\partial K$ embedded into $\mathring{M}_*$ and $\partial M_*$,
respectively.  We will refer to the former as a {\it boundary embedding}
and the latter as a {\it relative embedding}.  

For boundary embeddings, we naturally require that the isometry
$\iota_2\circ\iota_1^{-1}$ extends to an isomorphism of the boundary
normal bundles $\mathcal{N}(\iota_*(K))\subset T\partial M_*$. Under this assumption,
there is well-defined boundary connected sum along $K$, still denoted by
$M=M_1\#_KM_2$, see Section 2.2 for details.
\begin{thm1B}
Let $(M_1,g_1)$, $(M_2,g_2)$ be as in Theorem 1$_a$ and suppose
$(K,g_K)$ is a closed manifold with isometric embeddings
$\iota_1:K\to\partial M_1$, $\iota_2:K\to \partial M_2$ with
$m=n-k\geq3$.  Assume that $\iota_2\circ\iota_1^{-1}$ extends to an
isomorphism of the normal bundles
$\mathcal{N}(\iota_*(K))\subset T\partial M_*$.  Then
there exists a family of
scalar-flat metrics
$\{\tilde{g}_\varepsilon\}_{\varepsilon\in(0,\varepsilon_0)}$ with constant
boundary mean curvature
$H_{\tilde{g}_\varepsilon}=\mathcal{O}(\varepsilon^{m-2})$.  Moreover, the
metrics $\tilde{g}_{\varepsilon}$ are conformal to $g_*$ away from a
fixed tubular neighborhood of $\iota_*(K)$ in $M_*$ and converge to
the original metrics on compact sets of $M_*\setminus\iota_*(K)$ in
the $\mathcal{C}^2$ topology as $\varepsilon\to 0$ for $*=1,2$.
\end{thm1B}
\begin{thm2B}
  Assume, in addition to the conditions in Theorem 1$_b$, that both
  manifolds $(M_1,g_1)$ and $(M_2,g_2)$ are not Ricci-flat with
  vanishing second fundamental form of their boundaries.  Then
  there exists a second family of scalar-flat metrics
  $\{\hat{g}\}_{\varepsilon\in(0,\varepsilon_0)}$ on $M=M_1\#_KM_2$ with
  vanishing boundary mean curvature.  Moreover, $\hat{g}_\varepsilon\to
  g_{*}$ on compact sets of $M_*\setminus\iota_*(K)$ in the
  $\mathcal{C}^2$ topology as $\varepsilon\to0$ for $*=1,2$. 
\end{thm2B}
The construction for a relative embedding, however, is a bit more
delicate and we require additional assumptions on the embeddings
$\iota_*$.
\begin{definition}\label{surgery-ready}
We say that the isometric embeddings $\iota_*: K\to M_*$, $*=1,2$, are
\emph{surgery-ready} if
\begin{enumerate}[(i)]
\item $\iota_*$ is a proper embedding, i.e.,
  $\iota_*(\mathring{K})\subset \mathring{M}_*$ and $\iota_*(\partial
  K)\subset \partial M_*$;
\item there is a neighborhood, $V\subset K,$ of $\partial K$ such that the embedding 
	$\iota_*(K)$ agrees with the $g_*$-exponential 
	map on $\iota_*(\partial K)$ (see Figure \ref{fig:relativesurgery});
\item the map $\iota_2\circ\iota_1^{-1}$ extends to an
  isomorphism of the normal bundles $\mathcal{N}_1(K)$,
  $\mathcal{N}_2(K)$ which restricts to an isomorphism of the boundary
  normal bundles $\mathcal{N}_1(\partial K)$, $\mathcal{N}_2(\partial K)$.
\end{enumerate}  
\end{definition}
Assuming the embeddings $\iota_*: K\to M_*$ are surgery-ready,
there is a well-defined generalized connected sum $M=M_1\#_KM_2$ along $K$,
see Section 2.3 for details.  Precisely, we have the following pair of theorems.
\begin{thm1C}
Let $(M_1,g_1)$, $(M_2,g_2)$ be as in Theorem 1$_a$ and $(K,g_K)$ be a
compact manifold with boundary. Assume $\iota_1:K\to M_1$,
$\iota_2:K\to M_2$ are surgery ready isometric embeddings as above
with $m=n-k\geq 3$.  Then there exists a family of scalar-flat metrics
$\{\tilde{g}_\varepsilon\}_{\varepsilon\in(0,\varepsilon_0)}$ on $M=M_1\#_KM_2$
with constant boundary mean curvature
$H_{\tilde{g}_\varepsilon}=\mathcal{O}(\varepsilon^{m-2})$.  Moreover, the
metrics $\tilde{g}_{\varepsilon}$ are conformal to $g_*$ away from a
fixed tubular neighborhood of $\iota_*(K)$ in $M_*$ and converge to
the original metrics on compact sets of $M_*\setminus\iota_*(K)$ in
the $\mathcal{C}^2$ topology as $\varepsilon\to0$ for $*=1,2$.
\end{thm1C}

\begin{thm2C}
  Assume, in addition to the conditions in Theorem 1$_c$, that both
  manifolds $(M_1,g_1)$ and $(M_2,g_2)$ are not Ricci-flat with
  vanishing second fundamental form of their boundaries.  Then
  there exists a second family of scalar-flat metrics
  $\{\hat{g}_{\varepsilon}\}_{\varepsilon\in(0,\varepsilon_0)}$ on $M=M_1\#_KM_2$ with
  vanishing boundary mean curvature.  Moreover, $\hat{g}_\varepsilon\to
  g_{*}$ on compact sets of $M_*\setminus\iota_*(K)$ in the
  $\mathcal{C}^2$ topology as $\varepsilon\to0$ for $*=1,2$. 
\end{thm2C}
Before we begin, the author would like to thank his Ph.D. adviser, 
Prof. Boris Botvinnik for suggesting this problem as well as Prof.
Micah Warren for a number of helpful conversations. 

\section{The Yamabe problem for manifolds with boundary}
Let us introduce the objects and notations we will require.
For a smooth Riemannian $n$-dimensional manifold $(M,g)$ with 
boundary $\partial M$, we will write $Ric_g$ for its Ricci tensor and $A_g$
 for the second fundamental form of the boundary with respect to the 
outward unit normal vector $\nu$.
The scalar curvature of $(M,g)$ is given by $R_g=tr_gRic_g$ 
and its boundary mean curvature is $H_g=tr_gA_g$. Notice that $H_g$
is the sum of the principle curvatures at a point $p\in\partial M$, as opposed to their 
average (usually denoted by $h_g$) which is used in Escobar's original work
 \cite{E3}\cite{E2}\cite{E1}.

A metric $\tilde{g}$ is said to be {\em conformal} to $g$ if 
there is a smooth positive function $f$ so that $\tilde{g}=fg$. The 
equivalence class of metrics
conformal to $g$ will be denoted by $[g]$.  
We will often write the conformal factor in the form $f=\psi^{\frac{4}{n-2}}$.
Writing $c_n=\frac{n-2}{4(n-1)}$, 
the scalar curvature of $\tilde{g}=\psi^{\frac{4}{n-2}}g$ is given by
\[
R_{\tilde{g}}=\frac{L_g\psi}{c_n\psi^{\frac{n+2}{n-2}}}
\]
where $L_g$ is the conformal Laplacian $L_g=-\Delta_g+c_nR_g.$
The mean curvature of the boundary with respect to $\tilde{g}$ is given by
\[
H_{\tilde{g}}=\frac{B_g\psi}{2c_n\psi^{\frac{n}{n-2}}}
\]
where the first-order boundary operator $B_g$ is given by 
$B_g=\partial_\nu+2c_nH_g$ on $\partial M.$

In \cite{E3} Escobar studied and answered the following question: 
Does a given conformal class $[g]$ contain a 
scalar-flat metric with constant boundary mean curvature? In light of the above 
formula, this task is equivalent to solving the following elliptic problem 
with non-linear boundary conditions
\begin{equation}
\label{eq:RYP2}
\begin{cases}
\Delta_g\psi=c_nR_g\psi&\text{ in }M\\
\partial_\nu\psi=2c_n(Q \psi^{\frac{n}{n-2}}-H_g\psi)&\text{ on }\partial M
\end{cases}
\end{equation}
where $Q$ is a constant. If $\psi$ is a smooth solution to (\ref{eq:RYP2}), then 
$\tilde{g}=\psi^{\frac{4}{n-2}}g$ will have vanishing scalar curvature and constant
boundary mean curvature $\lambda$.
As mentioned above, equation (\ref{eq:RYP2}) 
is the Euler-Lagrange equation for the total scalar curvature plus total mean curvature functional (cf. \cite{Ar}), renormalized with respect to the volume of the boundary. In terms of the conformal factor $\psi$, this functional takes the form
\begin{align}
Q(\psi)=\frac{\int_M(|\nabla\psi|^2_g+c_nR_g\psi^2)d\mu_g+2c_n\int_{\partial M}H_g\psi^2d\sigma_g}{(\int_{\partial M}|\psi|^{\frac{2(n-1)}{n-2}}d\sigma)^{\frac{n-2}{n-1}}}\notag
\end{align}
where $d\mu_g$ and $d\sigma_g$ denote the Riemannian measure on $M$ and 
$\partial M$ induced by $g$. 

\section{Construction of $g_\varepsilon$ and the local a priori estimate}
In this section, we construct the generalized connected sum $M=M_1\#_K M_2$
and define a family of metrics $\{g_\varepsilon\}_{\varepsilon\in(0,\frac12)}$ on $M$.  
At this point, it is convenient to consider the cases of interior, boundary, and relative
embeddings separately.
The next step is to give pointwise and integral estimates for the scalar and 
boundary mean curvatures
of the new metrics $\{g_\varepsilon\}_{\varepsilon\in(0,\frac12)}$ 
cf. Propositions $1_a,1_b,$ and $1_c$. Finally, we study the family of 
operators $\Delta_{g_\varepsilon}$, giving a local a priori estimate for
solutions of the $\Delta_{g_\varepsilon}$-Poisson equation cf. Propositions
$2_a, 2_b,$ and $2_c$.

In section 2.1 we describe the process for interior embeddings, revisiting 
the construction in \cite{M1}.
In this case, the $g_*$-exponential map identifies, for some small $r>0$,
the distance neighborhood 
\[
V_*^r:=\{y\in M_*\colon\mathrm{dist}_{g_*}(y,\iota_*(K))<r\}
\]
with the 
portion of the normal bundle $\{w\in\mathcal{N}_*(K)\colon ||w||_{g_*}<r\}$. 
On $V_*^r$, these Fermi coordinates yield
good asymptotic expressions for the metric tensor $g_*$. 
These local expressions are then used to transition from $g_1$ to $g_2$
on annular regions about $\iota_1(K)$ and $\iota_2(K)$, in turn yielding a 
globally-defined metric $g_\epsilon$ on the sum, $M$, 
for each $\varepsilon\in(0,\frac12)$.

In the case of boundary and relative embeddings, however, there are two sorts of geodesics 
which must be used to visit all of the neighborhood 
$V_*^r$ from $\iota_*(K)$ --
those of $g_*$ and those of $g_*|_{\partial M_*}$. 
This complicates matters and we must provide new geometric constructions 
and estimates for a Poisson problem with mixed Dirichlet-Neuman boundary 
conditions. This analysis for boundary and relative embeddings
is carried out in sections 2.2 and 2.3, respectively.

\subsection{Interior embeddings}
Throughout this section we will only consider the case of interior embeddings; 
when $K$ is closed and embedded entirely 
within the interior $\mathring{M}_*$. By uniformly rescaling the metrics $g_1$ and $g_2$,
we may assume that 
\[
\exp^{g_*}:\{w\in\mathcal{N}_*(K)\colon||w||_{g_*}<1\}\to M_*
\]  
is a diffeomorphism onto its image.
For a fixed $\varepsilon\in(0,\frac12)$, we will give a local description  
of a gluing metric $g_\varepsilon$ on the disjoint union 
\[
\left(M_1\setminus V_1^{\varepsilon^2}\right)
\sqcup 
\left(M_2\setminus V_2^{\varepsilon^2}\right).
\]
This description will, in fact, immediately yield a globally defined metric $g_\varepsilon$
on the above disjoint union.  
We will then construct the connected sum $M_1\#_K M_2$ in such a way so that the metric 
$g_\varepsilon$ descends to it.

Let $U\subset K$ be a trivializing neighborhood for the normal bundles 
$\mathcal{N}_1(K)$ and 
$\mathcal{N}_2(K)$ with local coordinates
$z=(z^1,\dots,z^k)$.  Denote the open unit $m$-ball by
\[
D^m=\{x=(x^1,\dots,x^m)\in\mathbb{R}^m\colon|x|<1\}.
\]
The map
\[
F_*:U\times D^m\to M_*,\quad F_*(z,x):=\exp^{g_*}_{\iota_*(z)}(x)
\]
gives Fermi coordinates $(z,x)$ on a neighborhood of $\iota_*(U)$ in
$M_*$ for $*=1,2$.
Abusing notations, we write $(z,x)$ for the
coordinates on both $M_1, M_2$ and suppress the use of the bundle isomorphism
in identifying the trivializations over $U$.
These coordinates give the following local expression for the metric $g_*$
\[
g_*=g^{(*)}_{ij}dz^i dz^j+g^{(*)}_{i\alpha}dz^i dx^\alpha+
g^{(*)}_{\alpha\beta}dx^\alpha dx^\beta
\]
with the well-known expansions
\[
g^{(*)}_{ij}(z,x)=g^K_{ij}(z)+\mathcal{O}(|x|),\quad
g^{(*)}_{i\alpha}(z,x)=\mathcal{O}(|x|),\quad 
g^{(*)}_{\alpha\beta}(z,x)=\delta_{\alpha\beta}+\mathcal{O}(|x|^2).
\]  
Setting $x=\varepsilon e^{-t}\theta$ on $M_1$ and $x=\varepsilon
e^t\theta$ on $M_2$, we introduce modified polar coordinates
$(z,t,\theta)$ on a neighborhood about $\iota_*(U)$ in $M_*$ for $*=1,2$
where $\theta=(\theta^1,\dots,\theta^{m-1})$ are spherical
coordinates for the unit sphere $S^{m-1}$ and
$t\in(\log\varepsilon,-\log\varepsilon)$.  Notice that $t$ ranges
between the values $\log\varepsilon$ and $-\log\varepsilon$ as $|x|$
ranges between $\varepsilon^2$ and $1$.  We define two functions
$u_\varepsilon^{(1)},u_\varepsilon^{(2)}:(\log\varepsilon,-\log\varepsilon)\to\mathbb{R}$ by
 \[
u_\varepsilon^{(1)}(t):=\varepsilon^{\frac{m-2}{2}}e^{-\frac{m-2}{2}t}\quad\text{ and } 
	\quad u_\varepsilon^{(2)}(t):=\varepsilon^{\frac{m-2}{2}}e^{\frac{m-2}{2}t}.
\]
Using the coordinates $(z,t,\theta)$, the local expression for $g_*$ can be reorganized in the form
\[
\begin{array}{rr}
g_*=&g^{(*)}_{ij}dz^i dz^j+\left(u_\varepsilon^{(*)}\right)^{\frac{4}{m-2}}\Big{(}g^{(*)}_{tt}dt^2+
	g_{\lambda\mu}^{(*)}d\theta^\lambda d\theta^\mu+g_{t\lambda}^{(*)}dtd\theta^\lambda\Big{)}\\
{}&+g_{it}^{(*)}dz^i dt+g_{i\lambda}^{(*)}dz^i d\theta^\lambda.
\end{array}
\]
The asymptotics now take the form
\[
\begin{array}{lll}
g^{(*)}_{ij}(z,t,\theta)=g^K_{ij}(z)+\mathcal{O}(|x|), 
	& g^{(*)}_{\lambda\mu}(z,t,\theta)=g^{(\theta)}_{\lambda\mu}(\theta)+\mathcal{O}(|x|),
	& g_{tt}^{(*)}(z,t,\theta)=1+\mathcal{O}(|x|^2)\\
g^{(*)}_{i\lambda}(z,t,\theta)=\mathcal{O}(|x|^2), 
	& g^{(*)}_{it}(z,t,\theta)=\mathcal{O}(|x|^2),
	& g^{(*)}_{i\lambda}(z,t,\theta)=\mathcal{O}(|x|^2)
\end{array}
\]
where $g^{(\theta)}_{\lambda\mu}$ denotes a component of the standard
round metric on the unit sphere $S^{m-1}$ in the spherical coordinates
$(\theta^1,\dots,\theta^{m-1})$.

We are now ready to perform the interpolation between $g_1$ and $g_2$.  
Fix a cut-off smooth function $\xi:(\log\varepsilon,-\log\varepsilon)\to[0,1]$
which is non-increasing and takes the value 1 on $(\log\varepsilon,-1]$ and 0 on
$[1,-\log\varepsilon)$.  Similarly, let $\eta:(\log\varepsilon,-\log\varepsilon)\to[0,1]$ 
be a non-increasing, smooth function which takes 
the value 1 on $(\log\varepsilon,-\log\varepsilon-1]$ 
and the value 0 on $(-\log\varepsilon-\frac12,-\log\varepsilon)$. 
\begin{figure}[htb!]
\begin{center}
\begin{picture}(0,0)
\put(100,50){$\xi$}
\put(220,50){$\eta$}
\put(5,-5){$\log\varepsilon$}
\put(100,-5){$-1$}
\put(140,-5){$1$}
\put(240,-5){$-\log\varepsilon$}
\end{picture}
\includegraphics[height=3.2cm]{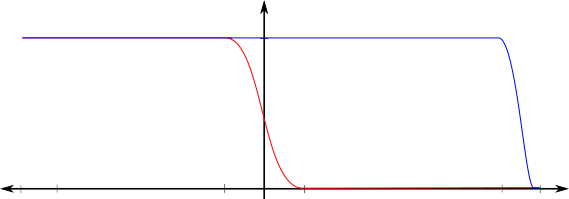}
\end{center}
\caption{The cut-off functions $\xi$ and $\eta$}
\label{fig:xieta}
\end{figure}
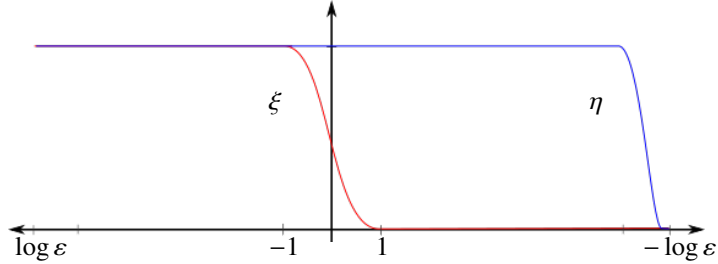

\noindent
Define a function 
$u_\varepsilon:(\log\varepsilon,-\log\varepsilon)\to\mathbb{R}$ by  
\[
u_\varepsilon(t)=\eta(t)u_\varepsilon^{(1)}+\eta(-t)u_\varepsilon^{(2)}.
\]
Finally, for each $\varepsilon\in(0,\frac12)$, define a metric $g_\varepsilon$ by
\[
\begin{array}{rr}
g_\varepsilon(z,t,\theta)=&(\xi g_{ij}^{(1)}+(1-\xi)g_{ij}^{(2)})dz^i dz^j+
	u_\varepsilon^{\frac{4}{n-2}}\Big{(} (\xi g^{(1)}_{tt} +(1-\xi)g^{(2)}_{tt})dt^2\\
{}&+(\xi g_{\lambda\mu}^{(1)}+
	(1-\xi)g_{\lambda\mu}^{(2)})d\theta^\lambda d\theta^\mu+ (\xi g_{t\lambda}^{(1)}+
	(1-\xi)g_{t\lambda}^{(2)})dtd\theta^\lambda]\Big{)}\\
{}&+(\xi g_{it}^{(1)}+(1-\xi)g_{it}^{(2)})dz^idt+(\xi g_{i\lambda}^{(1)}+
	(1-\xi)g_{i\lambda}^{(2)})dz^i d\theta^\lambda.
\end{array}
\]
This defines a metric $g_\varepsilon$ on the tubular annuli 
\[
V_*^1\setminus\overline{V_*^{\varepsilon^2}}=
	\{y\in M_*|\varepsilon^2<\mathrm{dist}_{g_*}(y,\iota_*(K))<1\}
\]
for $*=1,2$. 
We set $g_\varepsilon=g_*$ on $M_*\setminus \overline{V_*^1}$.
This gives well-defined metric $g_\varepsilon$ on the disjoint union
$(M_1\setminus V_*^{\varepsilon^2}) \sqcup
(M_2\setminus V_*^{\varepsilon^2})$.

Now we are ready to describe the generalized connected sum $M=M_1\#_K M_2$.
See Figure \ref{fig:boundarysurgery} for a picture in the boundary embedding case.
Let $\Phi:\mathcal{N}_1(K)\to\mathcal{N}_2(K)$ be the isomorphism of the
normal bundles given in the hypothesis of Theorem $1_a$. 
For each $\varepsilon\in(0,\frac12)$, consider the auxiliary fiber-wise mapping 
$\Psi_\varepsilon$ given by
\begin{align}
\Psi_\varepsilon&:\left(\mathcal{N}_1(K)\setminus\{0\}\right)\sqcup\left(\mathcal{N}_2(K)\setminus\{0\}\right)\to
\left(\mathcal{N}_1(K)\setminus\{0\}\right)\sqcup\left(\mathcal{N}_2(K)\setminus\{0\}\right)\notag\\
\Psi_\varepsilon&(z,t,\theta):=
\begin{cases}
\Phi(z,-t,\theta)&\text{ if }(z,t,\theta)\in\mathcal{N}_1(K)\\
\Phi^{-1}(z,-t,\theta)&\text{ if }(z,t,\theta)\in\mathcal{N}_2(K).
\end{cases}\notag
\end{align}
Notice that, in the Fermi coordinates $(z,x)$, this mapping can be expressed as
$\Psi_\varepsilon(z,x)=\Phi_\varepsilon(z,\frac{\varepsilon^2}{|x|^2}x)$.
We define
\[
M_\varepsilon:=\left((M_1\setminus V_*^{\varepsilon^2}) \sqcup
(M_2\setminus V_*^{\varepsilon^2})\right) /\sim_\varepsilon
\]
where we introduce the equivalence relation $\sim_\varepsilon$
on the disjoint union
\[
\left(V_1^1\setminus \overline{V_1^{\varepsilon^2}}\right)\sqcup\left(V_2^1\setminus \overline{V_2^{\varepsilon^2}}\right)
\]
as follows:
If $y\in V_1^1\setminus \overline{V_1^{\varepsilon^2}}$, then
$y\sim_\varepsilon (F_2\circ\Psi_\varepsilon\circ F_1^{-1})(y)$. 

Observing that $g_\varepsilon$ is invariant under $\Psi_\varepsilon$, the 
metric descends to $M_\varepsilon$. We will continue to denote this metric by 
$g_\varepsilon$. Since its diffeomorphism type does
not depend on $\varepsilon$, we will drop the subscript when referring to 
the generalized connected sum and simply write $M=M_\varepsilon$. This finishes 
the definition of the family of Riemannian manifolds $(M,g_\varepsilon)$.
The coordinates $(z,t,\theta)$ which were originally used on $M_1$ will
continue to be used as coordinates on $M$. 
We will require a piece of notation for certain subsets of the gluing region in $M$:
For each $\varepsilon>0$ and $a,b\geq 0$, we denote by
\[
T^\varepsilon(a,b)=\{(z,t,\theta)\in M\colon\log\varepsilon+a\leq t\leq-\log\varepsilon-b\}.
\]

Before we approach the problem of producing a solution to the system 
(\ref{eq:RYP2}) on $(M,g_\varepsilon)$, we will require two geometrical 
properties of the family $\{g_\varepsilon\}_{\varepsilon\in(0,\frac12)}$. 
In the present case of interior embeddings, these properties are
identical to those found in \cite{M1}.
Propositions $1_a$ and $2_a$ summarize the results of
\cite[Section 4]{M1}.
\begin{prop1A}{\rm (cf. \cite[Proposition 2]{M1})}
There is a constant $C>0$ such that
\[
|R_{g_\varepsilon}|\leq C\varepsilon^{-1}\cosh^{1-m}(t)
\]
on $T^\varepsilon(0,0)$ and
\[
\int_M|R_{g_\varepsilon}|d\mu_{g_\varepsilon}=\mathcal{O}(\varepsilon^{m-2}).
\]
Moreover, the constant $C$ depends only on $(K,g_K),(M_1,g_1),$ and $(M_2,g_2)$.
\end{prop1A}

The other feature of $g_\varepsilon$ we will need is an $\varepsilon$-uniform 
a priori estimate for solutions of the $\Delta_{g_\varepsilon}$-Poisson equation on the neck.
Indeed, the family of operators $\{\Delta_{g_\varepsilon}\}_{\varepsilon\in(0,\varepsilon_0)}$ 
is not uniformly elliptic and the estimate is tailor made for the family of metrics $g_\varepsilon$. 
To state it, we will fix a family of 
weighting functions $\psi_\varepsilon:M\to\mathbb{R}$ satisfying
\[
\psi_\varepsilon=
	\begin{cases}
	\varepsilon \cosh(t)&\text{ on } T^\varepsilon(1,1)\\
	1&\text{ on } M\setminus T^\varepsilon(0,0)
	\end{cases}
\]
and varying smoothly between the values on
$T^\varepsilon(0,0)\setminus T^\varepsilon(1,1)\subset M$ (see Figure \ref{fig:psi}).
For a given parameter $\gamma\in(0,m-2)$ consider the following weighted Banach spaces
\[
\mathcal{C}^0_\gamma(M):=\{v\in\mathcal{C}^0(M)\colon
||v||_{\mathcal{C}^0_\gamma(M)}:=\sup_M|\psi^\gamma_\varepsilon v|<\infty\}.
\]
Note that, for fixed $\varepsilon,\gamma$, 
the two norms $||\cdot||_{\mathcal{C}^0_\gamma(M)}$ and $\sup_M|\cdot|$ are equivalent, though the equivalence is not uniform in $\varepsilon$.

\begin{prop2A}{\rm (cf. \cite[Proposition 4]{M1})}
Given $\gamma\in(0,m-2)$, there are constants
$\alpha_1,\alpha_2>0$ and $C>0$ satisfying the following statement for
all $\varepsilon\in(0,e^{-\max\{\alpha_1,\alpha_2\}})$.  If $v,f\in
\mathcal{C}^0(T^\varepsilon(\alpha_1,\alpha_2))$ satisfy
$\Delta_{g_\varepsilon}v=f$, then
\[
v\leq 
C\psi_\varepsilon^{-\gamma}\left(
	\sup_{T^\varepsilon(\alpha_1,\alpha_2)}|\psi^{\gamma+2}_\varepsilon f|+
	\sup_{\partial T^\varepsilon(\alpha_1,\alpha_2)}|\psi^\gamma_\varepsilon v|\right)
\]
pointwise on $T^\varepsilon(\alpha_1,\alpha_2)$ and
\[
||v||_{\mathcal{C}^0_\gamma(T^\varepsilon(\alpha_1,\alpha_2))}\leq 
	C\left(||f||_{\mathcal{C}^0_{\gamma+2}(T^\varepsilon(\alpha_1,\alpha_2))}+
	||v||_{\mathcal{C}^0_\gamma(\partial T^\varepsilon(\alpha_1,\alpha_2))}\right).
\]
Moreover, the constants $\alpha_1,\alpha_2,$ and $C$ depend only on 
$\gamma,(K,g_K),(M_1,g_1),$ and $(M_2,g_2)$.
\end{prop2A}
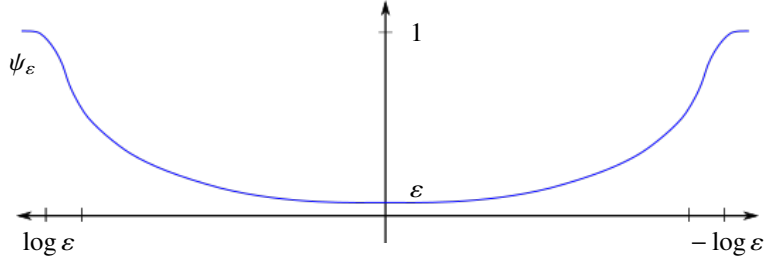
\begin{figure}[htb!]
\begin{center}
\begin{picture}(0,0)
\put(5,-3){$\log\varepsilon$}
\put(255,-3){$-\log\varepsilon$}
\put(150,76){$1$}
\put(0,65){$\psi_\varepsilon$}
\put(150,17){$\varepsilon$}
\end{picture}
\includegraphics[height=3.2cm]{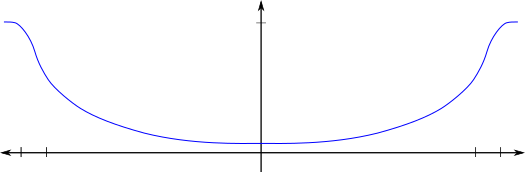}
\caption{The weighting function $\psi_\varepsilon$}
\label{fig:psi}
\end{center}
\vspace{-.5cm}
\end{figure}

\subsection{Boundary embeddings}%
In this section, we consider the setting of Theorems $1_b$ and $2_b$
-- when $\iota_*(K)$ lies entirely within $\partial M_*$.  As in section 2.1, 
we begin by defining the family of metrics $\{g_\varepsilon\}_{\varepsilon\in(0,\frac12)}$.
After uniformly rescaling the metrics $g_1$ and $g_2$, we may assume that both
\begin{align}
\exp^{g_*|_{\partial M_*}}:&\{w\in\mathcal{N}_*^\partial(K)\colon ||w||_{g_*}<1\}\to\partial M_*\notag\\
\exp^{g_*}:&\{w\in\mathcal{N}(\partial M_*)\colon ||w||_{g_*}<1\}\to M_*\notag
\end{align}
are diffeomorphisms onto their images for $*=1,2$.

Let $U\subset K$ be a trivializing neighborhood for the bundles $\mathcal{N}_1^\partial(K)$
and $\mathcal{N}_2^\partial(K)$ with local coordinates $z=(z^1,\dots,z^k)$.  The map
\[
F_*':U\times D^{m-1}\to\partial M_*,\quad F_*'(z,x'):=\exp^{g_*|_{\partial M_*}}_{\iota_*(z)}(x')
\]
gives Fermi coordinates $(z,x')$ for the boundary $\partial M_*$.
We denote the upper unit $m$-ball by 
\[
D^m_+:=\{(x',x^m)\in D^{m-1}\times\mathbb{R}\colon |(x',x^m)|<1 \text{ and } x^m\geq0\}.
\]
We identify the last component of $D^m_+$ with the inward normal
$\mathcal{N}(\partial M_*)$. Now the map
\[
F_*:U\times D^m_+\to M_*,\quad F(z,x',x^m):=\exp^{g_*}_{F_*'(z,x')}(x^m)
\]
gives coordinates $(z,x',x^m)$ on a neighborhood of $\iota_*(U)$ in $M_*$ for $*=1,2$.
We will write $x=(x',x^m)$ and $|x|:=\sqrt{|x'|^2+|x^m|^2}$.
In the coordinates $(z,x)$, the metric can be written as
\[
g_*=g^{(*)}_{ij}dz^i dz^j+g^{(*)}_{k\gamma}dz^k dx^\gamma+
g^{(*)}_{\alpha\beta}dx^\alpha dx^\beta
\]
with the following well-known expansions
\[
g^{(*)}_{ij}(z,x)=g^K_{ij}(z)+\mathcal{O}(|x|),\quad
g^{(*)}_{k\gamma}(z,x)=\mathcal{O}(|x|),\quad
g^{(*)}_{\alpha\beta}(z,x)=\delta_{\alpha\beta}+\mathcal{O}(|x|).
\]

We again introduce modified polar coordinates $(z,t,\theta)$ by 
setting $x=\varepsilon e^{-t}\theta$ on $M_1$ 
and $x=\varepsilon e^t\theta$ on $M_2$.
Here $\theta=(\theta^1,\dots,\theta^{m-1})$ are spherical coordinates on 
the unit upper hemisphere 
\[
S^{m-1}_+:=\{\theta\in S^{m-1}\colon 0\leq\theta^1\leq\frac{\pi}{4}\}
\]
and $t\in(\log\varepsilon,-\log\varepsilon)$. 
Notice that the boundary $\partial S^{m-1}_+$ can be identified with the set
$\{\theta\in S^{m-1}\colon \theta^1=\frac{\pi}{4}\}$.
Using the coordinates $(z,t,\theta)$, the local expression for $g_*$ can be reorganized in the form
\[
\begin{array}{rr}
g_*=&g^{(*)}_{ij}dz^i dz^j+\left(u_\varepsilon^{(*)}\right)^{\frac{4}{m-2}}\Big{(}g^{(*)}_{tt}dt^2+
	g_{\lambda\mu}^{(*)}d\theta^\lambda d\theta^\mu+g_{t\lambda}^{(*)}dtd\theta^\lambda\Big{)}\\
{}&+g_{it}^{(*)}dz^i dt+g_{i\lambda}^{(*)}dz^i d\theta^\lambda
\end{array}
\]
where $u_\varepsilon^{(*)}$ are defined as in section 2.1.  The asymptotics now take the form
\[
\begin{array}{lll}
g^{(*)}_{ij}(z,t,\theta)=g^K_{ij}(z)+\mathcal{O}(|x|), 
	& g^{(*)}_{\lambda\mu}(z,t,\theta)=g^{(\theta)}_{\lambda\mu}(\theta)+\mathcal{O}(|x|),
	& g_{tt}^{(*)}(z,t,\theta)=1+\mathcal{O}(|x|)\\
g^{(*)}_{i\lambda}(z,t,\theta)=\mathcal{O}(|x|), 
	& g^{(*)}_{it}(z,t,\theta)=\mathcal{O}(|x|),
	& g^{(*)}_{i\lambda}(z,t,\theta)=\mathcal{O}(|x|)
\end{array}
\]
where $g^{(\theta)}_{\lambda\mu}$ denotes a component of the standard round metric
on the upper unit hemisphere $S^{m-1}_+$ in the spherical coordinates $(\theta^1,\dots,\theta^{m-1})$.

Using the same cutoff functions $\xi$ 
and $\eta$ we introduced in the case of interior embeddings, define the function
$u_\varepsilon$ as in section 2.1. For each $\varepsilon\in(0,\frac12)$, set
\[
\begin{array}{rr}
g_\varepsilon(z,t,\theta)=&(\xi g_{ij}^{(1)}+(1-\xi)g_{ij}^{(2)})dz^i dz^j+
	u_\varepsilon^{\frac{4}{n-2}}\Big{(} (\xi g^{(1)}_{tt} +(1-\xi)g^{(2)}_{tt})dt^2\\
{}&+(\xi g_{\lambda\mu}^{(1)}+
	(1-\xi)g_{\lambda\mu}^{(2)})d\theta^\lambda d\theta^\mu+ (\xi g_{t\lambda}^{(1)}+
	(1-\xi)g_{t\lambda}^{(2)})dtd\theta^\lambda]\Big{)}\\
{}&+(\xi g_{it}^{(1)}+(1-\xi)g_{it}^{(2)})dz^idt+(\xi g_{i\lambda}^{(1)}+
	(1-\xi)g_{i\lambda}^{(2)})dz^i d\theta^\lambda.
\end{array}
\]
This defines a metric $g_\varepsilon$ on the tubular annuli 
$V_*^1\setminus\overline{V_*^{\varepsilon^2}}$ for $*=1,2$. 
We set $g_\varepsilon=g_*$ on $M_*\setminus \overline{V_*^1}$.
This gives well-defined metric $g_\varepsilon$ on the disjoint union
$(M_1\setminus V_1^{\varepsilon^2}) \sqcup(M_2\setminus V_2^{\varepsilon^2})$.

Now we are ready to describe the generalized connected sum $M=M_1\#_K M_2$.
See Figure \ref{fig:boundarysurgery} for a visual description.
Let $\Phi:\mathcal{N}_1(K)\to\mathcal{N}_2(K)$ be the isomorphism of the
normal bundles given in the hypothesis of Theorem $1_b$. 
For each $\varepsilon\in(0,\frac12)$, consider mapping 
$\Psi_\varepsilon$ given by
\begin{align}
\Psi_\varepsilon&:\left(\mathcal{N}_1(K)\setminus\{0\}\right)\sqcup\left(\mathcal{N}_2(K)\setminus\{0\}\right)\to
\left(\mathcal{N}_1(K)\setminus\{0\}\right)\sqcup\left(\mathcal{N}_2(K)\setminus\{0\}\right)\notag\\
\Psi_\varepsilon&(z,t,\theta):=
\begin{cases}
\Phi(z,-t,\theta)&\text{ if }(z,t,\theta)\in\mathcal{N}_1(K)\\
\Phi^{-1}(z,-t,\theta)&\text{ if }(z,t,\theta)\in\mathcal{N}_2(K).
\end{cases}\notag
\end{align}
We define
\[
M:=\left((M_1\setminus V_*^{\varepsilon^2}) \sqcup
(M_2\setminus V_*^{\varepsilon^2})\right) /\sim_\varepsilon
\]
where we introduce equivalence relation $\sim_\varepsilon$
on the disjoint union
\[
\left(V_1^1\setminus \overline{V_1^{\varepsilon^2}}\right)\sqcup\left(V_2^1\setminus \overline{V_2^{\varepsilon^2}}\right)
\]
as follows:
If $y\in V_1^1\setminus \overline{V_1^{\varepsilon^2}}$, then
$y\sim_\varepsilon (F_2\circ\Psi_\varepsilon\circ F_1^{-1})(y)$. 

Observing that $g_\varepsilon$ is invariant under $\Psi_\varepsilon$, the 
metric descends to $M$. This finishes 
the definition of the family of Riemannian manifolds $(M,g_\varepsilon)$.
\begin{figure}[htb!]
\begin{center}
  \begin{picture}(0,0)
\put(10,260){$M_1$}
\put(360,260){$M_2$}
\put(20,70){\large{$M$}}
\put(30,150){\small{$t=\log\varepsilon$}}
\put(25,5){$\partial M$}
\put(180,15){$T^\varepsilon(\alpha_1,\alpha_2)$}
\put(170,-10){$T^\varepsilon(0,0)$}
\put(100,140){\small{$t=\log\varepsilon+\alpha_1$}}
\put(182,130){\small{$t=0$}}
\put(230,140){$t=-\log\varepsilon-\alpha_2$}
\put(315,150){$t=-\log\varepsilon$}
\end{picture}
\includegraphics[height=4in]{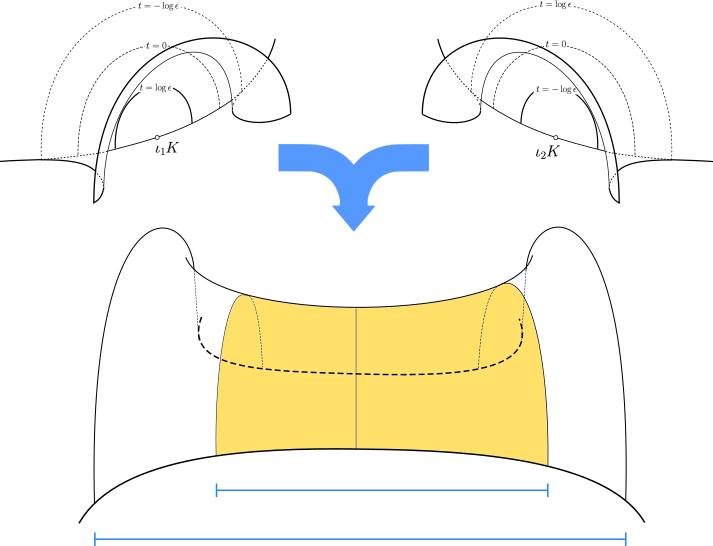}
\end{center}
\caption{The construction of $(M,g_{\epsilon})$ and the neck region 
	$T^\varepsilon(\alpha_1,\alpha_2)$}
\label{fig:boundarysurgery}
\end{figure}
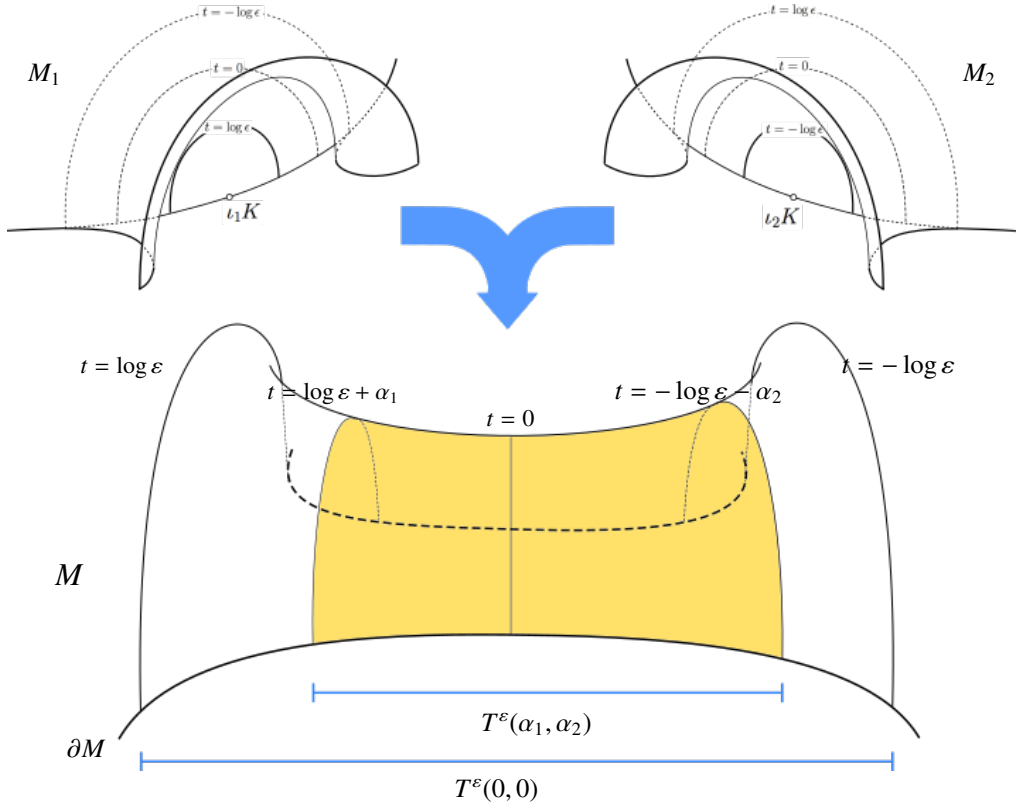
\subsubsection{The scalar and boundary mean curvatures of $g_\varepsilon$}
The next step is to produce analogs of propositions $1_a$ and $2_a$ for the 
case of boundary embeddings.  In addition to the
estimate for the scalar curvature $R_{g_\varepsilon}$, 
we will require a similar estimate for the
boundary mean curvature $H_{g_\varepsilon}$. 
\begin{prop1B}
There is a constant $C>0$, independent of $\varepsilon$, such that
\[
|R_{g_\varepsilon}|\leq C\varepsilon^{-1}\cosh^{1-m}(t),\quad |H_{g_\varepsilon}|\leq C\cosh^{2-m}(t)
\]
on $T^\varepsilon(0,0)$ and
\[
\int_M|R_{g_\varepsilon}|d\mu_{g_\varepsilon}=\mathcal{O}(\varepsilon^{m-2}),\quad\int_{\partial M}|H_{g_\varepsilon}|d\sigma_{g_\varepsilon}=\mathcal{O}(\varepsilon^{m-2}).
\]
\end{prop1B} 
\begin{proof}
The estimate on $R_{g_\varepsilon}$ can be obtained by an argument identical to the one found in \cite{M1} so we will only present the estimate on $H_{g_\varepsilon}$.

Let us first restrict our attention to 
the portion of $T^\varepsilon(0,0)$ where $\log \varepsilon +1\leq t\leq -1$.
On this portion of the neck the cut off function $\xi$ takes take the value 1 and 
$g_\varepsilon$ take the form 
\begin{align}
g_\varepsilon(z,x)=&g_{ij}^{(1)}(z,x)dz^i dz^j+
(1+\varepsilon^{m-2}|x|^{2-m})^{\frac{4}{m-2}}g_{\alpha\beta}^{(1)}(z,x)dx^\alpha dx^\beta 
\notag\\
{}&+g_{i\gamma}^{(1)}(z,x)dz^i dx^\gamma\notag.
\end{align}
We will drop the upper indices and write $g_{ij}=g^{(1)}_{ij}$, unless otherwise mentioned.

It will be useful to introduce a new formal parameter $\phi>0$ and introduce the following
two metrics on the neck $T^\varepsilon(0,0)$
\begin{align}
g(z,x,\phi)&=g_{ij}^{(1)}(z,x)dz^i dz^j+
(1+\phi)^{\frac{4}{m-2}}g_{\alpha\beta}^{(1)}(z,x)dx^\alpha dx^\beta 
+g_{i\gamma}^{(1)}(z,x)dz^i dx^\gamma\notag\\
\tilde{g}(z,\phi)&=g_{ij}^K(z)dz^idz^j+(1+\phi)^{\frac{4}{m-2}}\delta_{\alpha\beta}dx^\alpha dx^\beta\notag
\end{align}
If we choose $\phi=\varepsilon^{m-2}|x|^{2-m}$ in the formula for $g(z,x,\phi)$, observe that we
recover the gluing metric $g_\varepsilon$. Furthermore, we obtain the original metric
$g_1$ if we take $\phi=0$ in the formula for $g(z,x,\phi)$.
Our goal is to compute the boundary mean curvatures of the product metrics 
$\tilde{g}(z,\phi)$ and $\tilde{g}(z,0)$ 
then compare them to the corresponding curvatures of $g(z,x,\phi)$ and $g(z,x,0)$ in order to 
arrive at the desired estimate.

The Taylor expansions for the metric components now take the form
\[
g_{ij}(z,x,\phi)=\tilde{g}_{ij}(z,\phi)+\mathcal{O}(|x|),\quad
g_{\alpha\beta}(z,x,\phi)=\tilde{g}_{\alpha\beta}(z,\phi)+\mathcal{O}(|x|),\quad
g_{i\alpha}(z,x,\phi)=\mathcal{O}(|x|)
\]
Inspired by \cite{M1}, it will be convenient to adopt the following variant of big-o notation.
\begin{definition}
Let $a\in\mathbb{N}_0$ and let $f$ be a function of $z,x,$ and $\phi$. 
We say $f$ belongs to the class $\mathcal{A}_a$ if 
\[
|f(z,x,\phi)|\leq C|x|^a\quad\text{ and }\quad
|f(z,x,\phi)-f(z,x,0)|\leq C|x|^a|\phi|
\]
for some constant $C>0$.
\end{definition}
Notice that the product of an $\mathcal{A}_a$ function with an $\mathcal{A}_b$ function lies
in the class $\mathcal{A}_{a+b}$.
For the coefficients of the inverse of $g_\phi$, we may write
\[
g^{ij}(z,x,\phi)=\tilde{g}^{ij}(z,\phi)+\mathcal{A}_1,\quad
g^{\alpha\beta}(z,x,\phi)=\tilde{g}^{\alpha\beta}(z,\phi)+\mathcal{A}_1,\quad
g^{i\alpha}(z,x,\phi)=\mathcal{A}_1.
\]
Continuing, for any derivative of a component of $g(z,x,\phi)$, we have
\[
\partial_a g_{rs}(z,x,\phi)=
	\partial_a \tilde{g}_{rs}(z,\phi)+\mathcal{A}_0+|\nabla \phi|\mathcal{A}_1
\]
where $g_{rs}(z,x,\phi)$ may be any component of $g(z,x,\phi)$ 
in the coordinates $(z,x)$ and 
$\partial_a$ may be any derivative with respect to $z^i$ $(i=1,\dots,k)$ or 
$x^\alpha$ $(\alpha=1,\dots,m)$. 
Writing $\Gamma$ for a Christoffel symbol of $g(z,x,\phi)$ and 
$\tilde{\Gamma}$ for the 
corresponding symbol of $\tilde{g}(z,x)$, one may use the above computation with
the Kozul formula to find
\[
\Gamma=\tilde{\Gamma}+\mathcal{A}_0+|\nabla \phi|\mathcal{A}_1.
\]

Now consider the product metric $\tilde{g}(z,\phi)$. We have $H_{\tilde{g}(z,0)}=0$
since the boundary mean curvature of $(B^{m}_+(0),\delta_{\alpha\beta})$ vanishes.
Using the formula for boundary mean curvature under conformal change,
\begin{align}
H_{\tilde{g}(z,\phi)}&=\frac{1}{2c_n}(1+\phi)^{\frac{-m}{m-2}}\partial_{\nu}\phi\notag\\
{}&=-\frac{m-2}{2c_n}\lim_{x^m\to0}(1+\phi)^{\frac{-m}{m-2}}\varepsilon^{m-2}|x|^{-m}(x^m)\notag\\
{}&=0\notag,
\end{align}
where $x^m$ is the last coordinate of $x$.
Next we compute $H_{g(z,x,\phi)}$ in terms of $H_{\tilde{g}(z,\phi)}$ using the above expressions
for the Christoffel symbols
\begin{align}
H_{g(z,x,\phi)}&=g^{rs}(z,x,\phi)\Gamma^l_{rs}g_{lm}(z,x,\phi)\notag\\
{}&=(\tilde{g}^{rs}(z,\phi)+\mathcal{A}_1)(\tilde{\Gamma}^l_{rs}+
	\mathcal{A}_0+|\nabla \phi|\mathcal{A}_1)(\tilde{g}_{lm}(z,\phi)+\mathcal{A}_1)\notag\\
{}&=H_{\tilde{g}(z,\phi)}+\mathcal{A}_0+|\nabla \phi|\mathcal{A}_1\notag.
\end{align}
Taking $\phi=0$ in the above equation and subtracting from $H_{g(z,x,\phi)}$ yields
\[
|H_{g(z,x,\phi)}-H_{g(z,x,0)}|\leq|H_{\tilde{g}(z,\phi)}-H_{\tilde{g}(z,0)}|+C_1(|\phi|+|X||\nabla\phi|)
\]
for some positive constant $C_1$ independent of $\varepsilon$, coming from the definition 
of $\mathcal{A}_0$ and $\mathcal{A}_1$.
Now setting $\phi=\varepsilon^{m-2}|x|^{2-m}$ and recalling that $H_{\tilde{g}(z,\phi)}$
and $H_{\tilde{g}(z,0)}$ both vanish, we find
\[
|H_{g_\varepsilon}-H_{g_1}|\leq C_1e^{(m-2)t}
\]
concluding our work for $t\in(\log\varepsilon+1,-1)$.

Next, we move on to the portion $\{\log\varepsilon\leq t\leq \log\varepsilon+1\}$. 
On this part of the neck $\xi$ is still constant, but
the normal conformal factor $u_\varepsilon$ is effected by the cutoff function $\eta$.
However, since $\eta$ and its derivatives are uniformly bounded, it is straightforward 
to check that the estimate $|H_{g_{\varepsilon}}|\leq C_2e^{(m-2)t}$ holds here, 
where $C_2$ is a constant independent of epsilon.

On the portion of the neck $\{-1\leq t\leq0\}$, $\eta$ vanishes and now
the cutoff function $\xi$ effects all components of $g_\varepsilon$. However, we can still write
\begin{align}
g_\varepsilon(z,t,\theta)=&(g^{(1)}_{ij}+\mathcal{O}(|x|))dz^i dz^j+(1+\varepsilon^{m-2}|x|^{2-m})^{\frac{4}{m-2}}(g^{(1)}_{\alpha\beta}+\mathcal{O}(|x|))dx^\alpha dx^\beta\notag\\
{}&+(g^{(1)}_{k\gamma}+\mathcal{O}(|x|))dz^k dx^\gamma\notag.
\end{align}
In general, if two metrics are related by $g'=g+\mathcal{O}(|X|)$, we have 
$\Gamma'=\Gamma+\mathcal{O}(1)$ for any Christoffel symbol $\Gamma'$ of 
$g'$ and corresponding symbol $\Gamma$ of $g$. 
Hence the boundary mean curvatures satisfy
$|H_{g'}-H_g|=\mathcal{O}(1)$. Applying this fact to compare $g_\varepsilon$ and $g_1$,
we find that the mean curvature $H_{g_\varepsilon}$ is uniformly bounded in $\varepsilon$.
Since $t$ is small in absolute value
on this portion of the neck, we may choose $C_3>0$, independent of $\varepsilon$, so that
\[
|H_{g_\varepsilon}-H_{g_1}|\leq C_3e^{(m-2)t},
\]

To summarize our efforts, for $t\in(\log\varepsilon,0]$ and taking 
$C_4=\max(C_1,C_2,C_3)$, we have
\[
|H_{g_\varepsilon}-H_{g_1}|\leq C_4e^{(m-2)t}.
\]
Repeating these computations for the portion of the neck 
$\{0\leq t\leq-\log\varepsilon\}$, one can show that there is a 
constant $C_5$, independent of $\varepsilon$, satisfying
\[
|H_{g_\varepsilon}-H_{g_2}|\leq C_5e^{(2-m)t}
\]
for such $t$.
Recalling that $H_{g_*}\equiv0$ for $*=1,2$, these two inequalities 
give the pointwise estimate claimed in Lemma $1_b$ 
where the constant is given by $C=\max(C_4,C_5)$.

We conclude the proof by using our pointwise estimate to obtain 
the $L^1$ estimate on the boundary mean curvature 
\begin{align}
\int_{\partial M}|H_{g_\varepsilon}|d\sigma_{g_\varepsilon}&\leq 
	C\cdot\int_{\partial M\cap T^\varepsilon(0,0)}
	\cosh^{2-m}(t)d\sigma_{g_\varepsilon}\notag\\
{}&=C\cdot\mathrm{Vol}_{g_K}(K)\omega_{m-2}\varepsilon^{(m-2)}
	\int_{\log(\varepsilon)}^{-\log(\varepsilon)}e^{(2-m)t}\cosh^{(2-m)}(t)dt\notag\\
{}&\leq C'\cdot\mathrm{Vol}_{g_K}(K)\omega_{m-2}\varepsilon^{m-2}\notag
\end{align}
where $\omega_{m-2}$ denotes the volume of the unit sphere $S^{m-2}$ and $C'$ is 
another positive constant independent of $\varepsilon$.
\end{proof}

\subsubsection{Local Expression for $\Delta_{g_\varepsilon}$ and the Barrier Function $\phi_\delta$}
Before we can state our analogue of the a priori estimate Proposition $2_a$ for the 
boundary embedding case, we will need to construct a particular barrier function. 
First we define a function on the unit upper hemisphere $S_+^{m-1}$ in
spherical coordinates $\beta(\theta):=(L+1)-L\cos(\theta^1)$ where $L>0$ is a constant
to be determined. Notice that 
$\beta$ satisfies
\[
\begin{cases}
\Delta_{\theta}\beta(\theta)=-(m-1)L\cos(\theta^1)& \text{ in } S_+^{m-1}\\
\partial_{\theta^{1}}\beta(\theta)=\beta(\theta)&\text{ on } \partial S_+^{m-1}
\end{cases}
\]
and $1\leq\beta(\theta)\leq L+1$ in $S_+^{m-1}$. Now, for a fixed parameter 
$\delta\in(\frac{2-m}{2},\frac{m-2}{2})$, we define the function
on the gluing region by
\[
\phi_\delta(z,t,\theta):=\begin{cases}
	\frac{\cosh^\delta(t)}{u_\varepsilon(t)}\beta(\theta)&\text{ if }\delta\leq0\\
	\frac{\cosh(\delta t)}{u_\varepsilon(t)}\beta(\theta)&\text{ if }\delta\geq0
	\end{cases}
\]
which is a version of the barrier function used in \cite{M1}, modified for the
present case of boundary embeddings. The following lemma states 
the key properties of $\phi_\delta$ which we will need for the a priori estimate.
\begin{lemma}
Let $\delta\in(\frac{2-m}{2},\frac{m-2}{2})$. There exists a choice of
parameters $\alpha_1,\alpha_2>1$, $L>0$,
and a constant $C>0$ so that
\[
\begin{array}{rll}
\Delta_{g_\varepsilon}\phi_\delta&\leq
	-Cu_\varepsilon^{\frac{-4}{m-2}}\phi_\delta&
	\text{ in } T^\varepsilon(\alpha_1,\alpha_2)\\
\partial_\nu\phi_\delta&\geq 
	\frac12u_\varepsilon^{\frac{-2}{m-2}}\phi_\delta &
	\text{ on }\partial M\cap T^\varepsilon(\alpha_1,\alpha_2)
\end{array}
\]
is satisfied for all $\varepsilon\in(0,e^{-\max(\alpha_1,\alpha_2)})$.
\end{lemma}
\begin{proof}
Our first step is to obtain a 
useful local expression for the $g_\varepsilon$-Laplacian.
We will only need to consider the portion of the neck $T^\varepsilon(1,1)$ where the cut off function
 $\eta$ is constant and
the components of $g_\varepsilon$ take the form
\[
\begin{array}{ll}
g^\varepsilon_{ij}=g^K_{ij}+\mathcal{O}(|x|),& g^\varepsilon_{it}=\mathcal{O}(|x|^2)\notag\\
g^\varepsilon_{i\lambda}=\mathcal{O}(|x|^2),&
	g^\varepsilon_{tt}=u_{\varepsilon}^{\frac{4}{m-2}}(1+\mathcal{O}(|x|))\notag\\
g^\varepsilon_{t\lambda}=u^{\frac{4}{m-2}}_\varepsilon\mathcal{O}(|x|),&
	g^\varepsilon_{\lambda\mu}=u_\varepsilon^{\frac{4}{m-2}}(g^{(\theta)}_{\lambda\mu}+
	\mathcal{O}(|x|))\notag
\end{array}
\]
where $g^{(\theta)}_{\lambda\mu}$ denotes a component of the 
standard round metric on
the upper unit hemi-sphere $S^{m-1}_+$ in spherical coordinates 
$\theta=(\theta^1,\dots,\theta^{m-1})$.
As for the volume form, we have
\[
\sqrt{g_\varepsilon}=\sqrt{g_K}\sqrt{g_\theta}u_{\varepsilon}^\frac{2m}{m-2}(1+\mathcal{O}(|x|))
\]
where we write $\sqrt{g_\theta}=\sqrt{\det\left(g_{\lambda\mu}^{(\theta)}\right)}$. 
One can use the above expressions with Cramer's rule to compute the following 
expansions for components of the inverse matrix 
$g_\varepsilon^{-1}$
\[
\begin{array}{ll}
g_\varepsilon^{ij}=g_K^{ij}+\mathcal{O}(|x|),& g_\varepsilon^{it}=\mathcal{O}(|x|^2)\notag\\
g^{i\lambda}_\varepsilon=\mathcal{O}(|x|^2),&
	g_\varepsilon^{tt}=u_\varepsilon^{\frac{-4}{m-2}}(1+\mathcal{O}(|x|))\notag\\
g_\varepsilon^{t\lambda}=u_\varepsilon^{\frac{-4}{m-2}}\mathcal{O}(|x|),&
	g_\varepsilon^{\lambda\mu}=u_\varepsilon^{\frac{-4}{m-2}}
	(g^{\lambda\mu}_{(\theta)}+\mathcal{O}(|x|)).\notag
\end{array}
\]

Recall the following general fact: for a local coordinate system $y=(y^1,\dots,y^n)$ 
of a Riemannian manifold $(N,g)$, 
the $g$-Laplacian can be expressed as 
$\Delta_g\cdot=\frac{1}{\sqrt{g}}\partial_{y^a}(\sqrt{g}\;g^{ab}\partial_{y^b}\cdot)$.
Using this, a straight-forward computation gives us the following expression
\[
\Delta_{g_\varepsilon}=u_\varepsilon^{\frac{-4}{m-2}}\left(\partial_t^2+(m-2)\tanh\left(\frac{m-2}{2}t\right)\partial_t
+\Delta_{\theta}+u^{\frac{4}{m-2}}_\varepsilon\Delta_K+\mathcal{O}(|x|)\Phi_1\right)
\]
where $\Delta_{\theta}$ is the Laplace operator of the standard round metric on $S^{m-1}$, 
$\Delta_K$ is the Laplace operator of $(K,g_K)$, and $\Phi_1$ is a linear second-order 
operator with $\varepsilon$-uniformly bounded coefficients.
Now notice that one can conjugate $\Delta_{g_\varepsilon}$ by 
$u_\varepsilon$ to find
\begin{equation}\label{eq:laplace}
\Delta_{g_\varepsilon}\cdot=u_{\varepsilon}^{-\frac{m+2}{m-2}}\mathcal{D}_\varepsilon(u_\varepsilon \cdot)
\end{equation}
where $\mathcal{D}_\varepsilon$ is an operator of the form
\[
\mathcal{D}_\varepsilon=\partial_t^2-\left(\frac{m-2}{2}\right)^2+\Delta_{\theta}
+u_\varepsilon^{\frac{4}{m-2}}\Delta_K+\mathcal{O}(|x|)\Phi_2.
\]
In the above, $\Phi_2$ is another linear second order operator 
with $\varepsilon$-uniformly bounded coefficients.

Let us first consider the case $\delta\in(\frac{2-m}{2},0)$. One can use the 
conjugation formula (\ref{eq:laplace}) to find
\begin{align}
\Delta_{g_\varepsilon}\phi_\delta&=u_\varepsilon^{-\frac{m+2}{m-2}}\mathcal{D}_\varepsilon(\cosh^\delta(t)\beta(\theta))\notag\\
{}&=u_\varepsilon^{\frac{-4}{m-2}}\phi_\delta\left(\delta^2-\left(\frac{m-2}{2}\right)^2+\frac{(m-1)L\cos(\theta^1)}{\beta(\theta)}+\mathcal{O}(|x|)+(\delta-\delta^2)\cosh^{-2}(t)\right)\notag.
\end{align}
Evidently, we have $\delta-\delta^2\leq0$.  If we choose the positive constant
$L:=\frac{\left(\frac{m-2}{2}\right)^2-\delta^2}{m}$,
then the inequality
\begin{align}
\delta^2-\left(\frac{m-2}{2}\right)^2+\frac{(m-1)L\cos(\theta^1)}{\beta(\theta)}&\leq
\delta^2-\left(\frac{m-2}{2}\right)^2+(m-1)L\notag\\
{}&<0\notag
\end{align}
for all $\theta$.  Now, in order to deal with the above $\mathcal{O}(|x|)$ term in the expression 
for $\Delta_{g_\varepsilon}\phi^\partial_\delta$, observe that
we can find $\alpha_1,\alpha_2$ such that
\[
\delta^2-\left(\frac{m-2}{2}\right)^2+\frac{(m-1)L\cos(\theta^1)}{\beta(\theta)}+\mathcal{O}(|X|)\leq\frac12\left(\delta^2-
\left(\frac{m-2}{2}\right)^2+(m-1)L\right)
\]
on $T^\varepsilon(\alpha_1,\alpha_2)$ for all 
$\varepsilon\in(0,e^{-\max(\alpha_1,\alpha_2)})$. Now setting $C:=\frac12\left(\delta^2-
\left(\frac{m-2}{2}\right)^2+(m-1)L\right)$,
\[
\Delta_{g_\varepsilon}\phi_\delta\leq-Cu_\varepsilon^\frac{-4}{m-2}\phi_\delta
\]
on $T^\varepsilon(\alpha_1,\alpha_2)$. As similar argument for $\delta\in(0,\frac{m-2}{2})$ 
yields the desired estimate for $\Delta_{g_\varepsilon}\phi_\delta$.

Next, we consider the outward normal derivative of $\phi_\delta$. Recall the following 
general fact: if $\{\partial_{y^1},\dots,\partial_{y^{n-1}}\}$ span the boundary tangent 
space of a Riemannian manifold 
$(N,g)$ and $\partial_{y^n}$ points outwards, then the outward normal unit vector to $\partial N$
with respect to $g$ is given by the formula $\frac{g^{na}\partial_{y^a}}{\sqrt{g^{nn}}}$.
In our present situation, observe that 
$\{\partial_{z^1},\dots,\partial_{z^k},\partial_t,\partial_{\theta^1},\dots,\partial_{\theta^{m-2}}\}$
span the tangent space of $\partial M\cap T^\varepsilon(1,1)$ and $\partial_{\theta^1}$ points
outwards. Using this formula with the expressions for components of $g_\varepsilon^{-1}$, 
observe that the outward normal derivative on 
$\partial M\cap T^\varepsilon(1,1)$ with respect to $g_\varepsilon$ can be written as
\[
\partial_\nu=
	u_\varepsilon^{\frac{2}{m-2}}(u_\varepsilon^{-\frac{4}{m-2}}\partial_{\theta^1}+\mathcal{O}(|X|)\Phi_3)
\]
where $\Phi_3$ is a linear first-order differential operator on $\partial M\cap T^\varepsilon(1,1)$ 
with $\varepsilon$-uniformly bounded coefficients. 
Applying this to the barrier function $\phi_\delta$, we have
\[
\partial_\nu\phi_\delta=\phi_\delta u_\varepsilon^{-\frac{2}{m-2}}(1+\mathcal{O}(|x|)).
\]
By choosing yet larger $\alpha_1,\alpha_2$, we may assume that the above term 
satisfies $1+\mathcal{O}(|x|)\geq\frac12$.
we may assume
\[
\partial_\nu\phi_\delta\geq \frac12u_\varepsilon^{-\frac{2}{m-2}}\phi_\delta
\]
on $\partial M\cap T^\varepsilon(\alpha_1,\alpha_2)$ for all 
$\varepsilon\in(0,e^{-\max(\alpha_1,\alpha_2)})$, as claimed.
\end{proof}

\subsubsection{The local a priori estimate}
In order to state the a priori estimate, we will  
decompose the boundary of the region $T^\varepsilon(\alpha_1,\alpha_2)$ into two portions
$\partial T^\varepsilon(\alpha_1,\alpha_2)=\partial_1T^\varepsilon(\alpha_1,\alpha_2)\cup\partial_2 T^\varepsilon(\alpha_1,\alpha_2)$
where 
\begin{align}
\partial_1T^\varepsilon(\alpha_1,\alpha_2)&
	=\{(z,t,\theta)\in T^\varepsilon(\alpha_1,\alpha_2)\colon t=\log\varepsilon+\alpha_1
	\text{ or } t=-\log\varepsilon-\alpha_2\}\notag\\
\partial_2T^\varepsilon(\alpha_1,\alpha_2)&
	=\{(z,t,\theta)\in T^\varepsilon(\alpha_1,\alpha_2)\colon \theta^{1}=\frac{\pi}{2}\}.\notag
\end{align}
Note that $\partial_1T^\varepsilon(\alpha_1,\alpha_2)\subset M$, 
$\partial_2T^\varepsilon(\alpha_1,\alpha_2)\subset\partial M$, and the two meet at a corner.
\begin{prop2B}
Given $\gamma\in(0,m-2)$ there are $\varepsilon$-uniform constants 
$\alpha_1,\alpha_2>1$ and $C>0$ satisfying the following statement
for all $\varepsilon\in(0,e^{-\max\{\alpha_1,\alpha_2\}})$. 
If $v,f\in \mathcal{C}^0(T^\varepsilon(\alpha_1,\alpha_2))$ satisfy $\Delta_{g_\varepsilon}v=f$, then
\[
v\leq 
C\psi_\varepsilon^{-\gamma}\left(\sup_{T^\varepsilon(\alpha_1,\alpha_2)}|\psi^{\gamma+2}_\varepsilon f|+
\sup_{\partial_1 T^\varepsilon(\alpha_1,\alpha_2)}|\psi^\gamma_\varepsilon v|+
\sup_{\partial_2 T^\varepsilon(\alpha_1,\alpha_2)}|\psi^{\gamma+1}_\varepsilon\partial_\nu v|\right)
\]
pointwise on $T^\varepsilon(\alpha_1,\alpha_2)$ and 
\[
||v||_{\mathcal{C}^0_\gamma(T^\varepsilon(\alpha_1,\alpha_2))}\leq 
C\left(||f||_{\mathcal{C}^0_{\gamma+2}(T^\varepsilon(\alpha_1,\alpha_2))}+
||v||_{\mathcal{C}^0_\gamma(\partial_1 T^\varepsilon(\alpha_1,\alpha_2))}+
||\partial_\nu v||_{\mathcal{C}^0_{\gamma+1}(\partial_2 T^\varepsilon(\alpha_1,\alpha_2))}\right).
\]
\end{prop2B}
\begin{proof}
Set $\delta=\gamma-\frac{m-2}{2}$ and let $C',\alpha_1,\alpha_2$ 
be the constants given by Lemma 1.
Now consider the function
\[
\tilde{v}=a\phi_\delta-v
\]
where the constant $a>0$ is given by
\[
\begin{array}{rr}
a:=&\max(2,C'^{-1})\Big{(}\sup_{T^\varepsilon(\alpha_1,\alpha_2)}|u_\varepsilon^{\frac{4}{m-2}}\phi_\delta^{-1}f|+
\sup_{\partial_1T^\varepsilon(\alpha_1,\alpha_2)}|\phi_\delta^{-1}v|\\
{}&\quad+\sup_{\partial_2 T^\varepsilon(\alpha_1,\alpha_2)}|u_\varepsilon^{\frac{2}{m-2}}\phi_\delta^{-1}\partial_\nu v|\Big{)}.
\end{array}
\]
Our goal is to show that $\tilde{v}\geq0$.
First note that $\tilde{v}$ is superharmonic -- applying the inequalities of Lemma 1, we have
\begin{align}
\Delta_{g_\varepsilon}\tilde{v}&\leq-aC'u_\varepsilon^{\frac{-4}{m-2}}\phi_\delta-f\notag\\
{}&\leq -u_\varepsilon^{\frac{-4}{m-2}}\phi_\delta\sup_{T^\varepsilon_\alpha}
	|u_\varepsilon^{\frac{4}{m-2}}\phi_\delta^{-1}f|u-f\notag\\
{}&\leq0\notag.
\end{align}
Also observe that $\tilde{v}\geq0$ on $\partial_1T^\varepsilon(\alpha_1,\alpha_2)$. So
far, we have found 
\[
\begin{array}{rll}
\Delta_{g_\varepsilon}\tilde{v}&\leq0&\text{ in }T^\varepsilon(\alpha_1,\alpha_2)\\
\tilde{v}&\geq0&\text{ on } \partial_1 T^\varepsilon(\alpha_1,\alpha_2).
\end{array}
\]

The maximum principle for $\Delta_{g_\varepsilon}$ 
tells us the minimum of $\tilde{v}$ occurs somewhere on the boundary
of $T^\varepsilon(\alpha_1,\alpha_2)$.  
Suppose the minimum of $\tilde{v}$ occurs at a point 
$y_0\in \partial_2T^\varepsilon(\alpha_1,\alpha_2)$.
We may then apply the Hopf lemma and the estimate on $\partial_\nu\phi_\delta$ from Lemma 1 to obtain a contradiction
\begin{align}
0&> \partial_\nu\tilde{v}(y_0)\notag\\
{}&\geq aC'\phi_\delta u_\varepsilon^{\frac{-2}{m-2}}-\partial_\nu v(y_0)\notag\\
{}&\geq0\notag.
\end{align}
We conclude that the minimum of $\tilde{v}$ must occur on $\partial_1T^\varepsilon(\alpha_1,\alpha_2)$.
Since $\tilde{v}$ is non-negative there, $\tilde{v}\geq0$ on all of $T^\varepsilon(\alpha_1,\alpha_2)$.
In other words,
\begin{align}\label{eq:pest}
v\leq \max(2,C'^{-1})\phi_\delta\Big{(}\sup_{T^\varepsilon(\alpha_1,\alpha_2)}|u_\varepsilon^{\frac{4}{m-2}}\phi_\delta^{-1}f|&+
\sup_{\partial_1T^\varepsilon(\alpha_1,\alpha_2)}|\phi_\delta^{-1}v|\notag\\
{}&+\sup_{\partial_2 T^\varepsilon(\alpha_1,\alpha_2)}|u_\varepsilon^{\frac{2}{m-2}}\phi_\delta^{-1}\partial_\nu v|\Big{)}
\end{align}
on $T^\varepsilon(\alpha_1,\alpha_2)$. 

One can repeat the above argument, replacing $\tilde{v}$ with $a\phi_\delta+v$, 
to arrive at a similar lower bound on $v$. Together, we arrive at 
\begin{align}\label{eq:best}
\sup_{T^\varepsilon(\alpha_1,\alpha_2)}|\phi_\delta^{-1}v|\leq C'\Big{(}\sup_{T^\varepsilon(\alpha_1,\alpha_2)}|u_\varepsilon^{\frac{4}{m-2}}\phi_\delta^{-1}f|&+
\sup_{\partial_1T^\varepsilon(\alpha_1,\alpha_2)}|\phi_\delta^{-1}v|\notag\\
{}&+\sup_{\partial_2T^\varepsilon(\alpha_1,\alpha_2)}|u_\varepsilon^{\frac{2}{m-2}}\phi_\delta^{-1}\partial_\nu v|\Big{)},
\end{align}
noting that the constant $\max(2,C'^{-1})$ is independent of $\varepsilon$.

To phrase our estimate in terms of the weighted Banach spaces $\mathcal{C}^0_\gamma$,
we need to compare the functions $u_\varepsilon$ and $\phi_\delta$ to the 
weighting functions $\psi_\varepsilon$.
Recall the following basic fact of the hyperbolic cosine function: For every $\lambda>0$, 
there is a positive constant $C_\lambda$ so that
\[
C_\lambda^{-1}\cosh^\lambda(s)\leq\cosh(\lambda s)\leq C_\lambda\cosh^\lambda(s)
\]
holds for all $t\in\mathbb{R}$. For instance, 
recalling that $\psi_\varepsilon=\varepsilon\cosh(t)$ on 
$T^\varepsilon(\alpha_1,\alpha_2)$, 
there is a constant $C_\delta$ depending only on $\delta$ such that
\[
C_\delta^{-1}\psi_\varepsilon^{\frac{m-2}{2}-\delta}\leq\varepsilon^\delta\phi_\delta^{-1}
\leq C_\delta\psi_\varepsilon^{\frac{m-2}{2}-\delta}.
\]
Recalling that $\gamma=\frac{m-2}{2}-\delta$, one may replace $\phi_\delta^\partial$ 
and $u_\varepsilon$ with appropriate powers of $\psi_\varepsilon$ to 
reorganize the estimates (\ref{eq:pest}) and (\ref{eq:best}) to the one claimed in 
Lemma $2$ where 
$C=\max(2,C'^{-1},C_\delta)$.
\end{proof}
\subsection{Relative embeddings}
We will now consider the relative embedding case. Now $K$ itself 
has non-empty boundary $\partial K$.
Let $U\to\partial K$ be a coordinate chart for the boundary of $K$ with coordinates 
$z'=(z^1,\dots,z^{k-1})$ and, letting $z^k\in[0,1]$ be the inward normal direction, form
Fermi coordinates $z=(z',z^k)$ on a neighborhood of $U$ in $K$.
We will split the chart $U\times[0,3]$ into three parts 
\[
U^-:=U\times[0,1], \quad U^T:=U\times[1,2], \quad U^+:=U\times[2,3].
\]
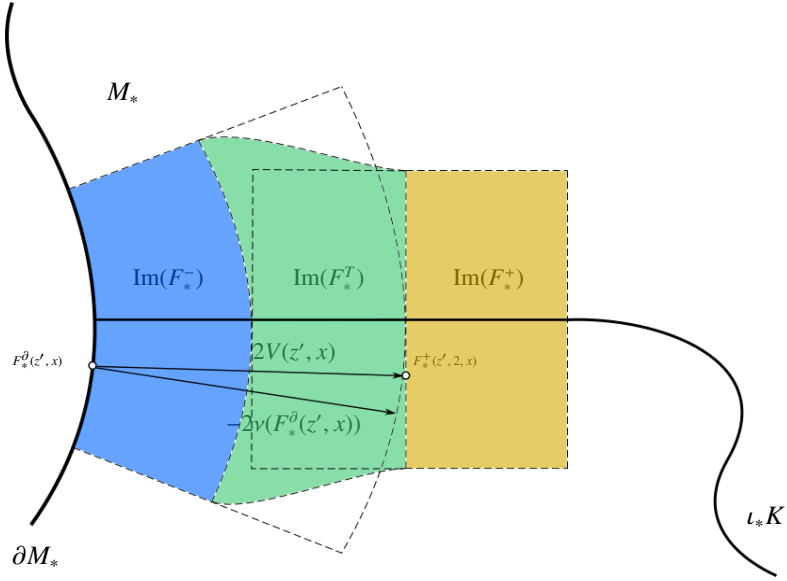
\begin{figure}[htb!]
\begin{center}
\begin{picture}(0,0)
\put(5,5){$\partial M_*$}
\put(40,180){$M_*$}
\put(280,20){$\iota_*K$}
\put(50,110){\small{$\mathrm{Im}(F_*^-)$}}
\put(110,110){\small{$\mathrm{Im}(F_*^T)$}}
\put(170,110){\small{$\mathrm{Im}(F^+_*)$}}
\put(5,80){\tiny{$F^\partial_*(z',x)$}}
\put(155,80){\tiny{$F^+_*(z',2,x)$}}
\put(95,82){\small{$2V(z',x)$}}
\put(85,55){\small{$-2\nu(F^\partial_*(z',x))$}}
\end{picture}
\includegraphics[height=3in]{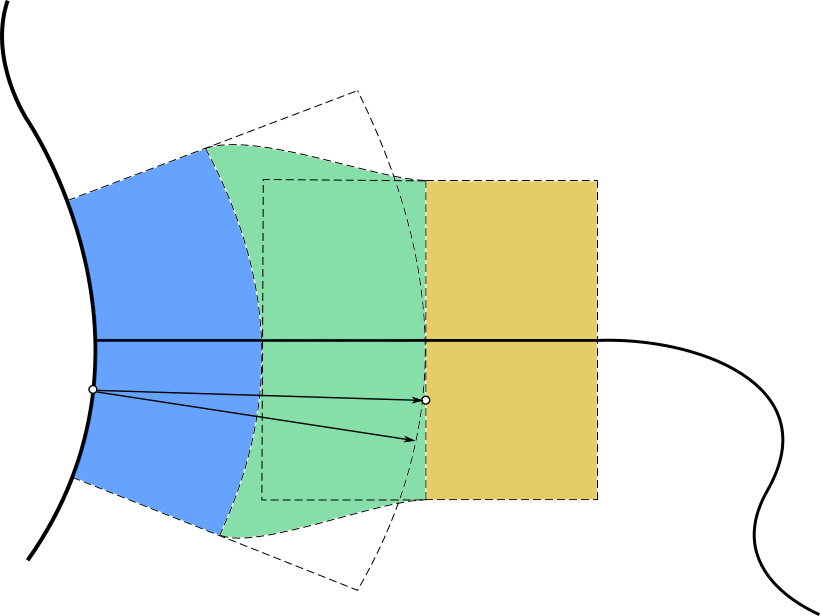}
\end{center}
\caption{The coordinate charts $F_*^-,F_*^T,F_*^+$ and the vector field $V$ compared to 
	the vector field $-\nu$}
\label{fig:relativesurgery}
\end{figure}
On $U^+$, we give Fermi coordinates given by
\[
F_*^+:U^+\times D^m\to M_*,\quad (z,x)\mapsto \exp^{g_*}_{\iota_*(z)}(x)
\]
which we originally saw in the interior embedding case from section 2.1.
As for $U^-$, we first have boundary Fermi coordinates $(z',x)$ for $\partial M_*$ given by
\[
F_*^\partial:U\times \{0\}\times D^{m}\to M_*,\quad (z',x)\mapsto \exp^{g_*|_{\partial M_*}}_{\iota_*(z')}(x).
\]
Now, similar to the boundary embedding construction from section 2.2, 
we get coordinates on $M_*$ by the mapping
\[
F_*^-:U^-\times D^m\to M_*,\quad (z',z^k,x)\mapsto \exp^{g_*}_{F_*^\partial(z',x)}(-z^k\nu),
\]
where $\nu$ is the outward-pointing normal vector to $\partial M_*$ with respect to $g_*$.
In order to transition between the two coordinate systems $F_*^-$ and $F_*^+$, we first define a vector
$V(z',x)\in T_{F^\partial_*(z',x)}M_*$ by solving the equation
\[
\exp^{g_*}_{F^\partial_*(z',x)}(2V(z',x))=F^+_*(z',2,x).
\]
Now we fix a non-increasing cutoff function $\alpha:[0,3]\to[0,3]$ 
which takes the value 1 on $[0,1]$ and 0 on $[2,3]$ and form a transitioning normal
vector by
\[
\overline{\nu}(z',z^k,x):=-\nu(F^\partial_*(z',x))\alpha(z^k)+(1-\alpha(z^k))V(z',x).
\]
The coordinate system on $U^T$ is given by the mapping
\[
F_*^T:U^T\times B^m\to M_*,\quad (z',z^k,x)\mapsto \exp^{g_*}_{F^\partial_*(z',x)}(z^k\overline{\nu}(z',z^k,x)).
\]
Noting that $F^+_*=F^T_*$ when $z^k=2$, $F^-_*=F^T_*$ when $z^k=1$, and $z=(z',z^k)$, we have
well-defined coordinates $(z,x)$ on a neighborhood of the boundary of $\iota_*(K)$ in $M_*$
(see Figure \ref{fig:relativesurgery}).
As for an interior neighborhood of $\iota_*(K)$, we have the Fermi coordinates from section 2.1
and refer to both coordinate systems with $(z,x)$.

On either interior or boundary charts, we introduce the coordinates $(z,t,\theta)$ by 
setting $x=\varepsilon e^{-t}\theta$ on $M_1$ 
and $x=\varepsilon e^t\theta$ on $M_2$.
Here $\theta=(\theta^1,\dots,\theta^{m-1})$ are spherical coordinates on 
the unit sphere $S^{m-1}$ and $t\in(\log\varepsilon,-\log\varepsilon)$. 
The metric $g_*$ can be expressed in the form
\[
\begin{array}{rr}
g_*=&g^{(*)}_{ij}dz^i dz^j+\left(u_\varepsilon^{(*)}\right)^{\frac{4}{m-2}}\Big{(}g^{(*)}_{tt}dt^2+
	g_{\lambda\mu}^{(*)}d\theta^\lambda d\theta^\mu+g_{t\lambda}^{(*)}dtd\theta^\lambda\Big{)}\\
{}&+g_{it}^{(*)}dz^i dt+g_{i\lambda}^{(*)}dz^i d\theta^\lambda
\end{array}
\]
where $u_\varepsilon^{(*)}$ is defined as in section 2.1.  The asymptotics now take the form
\[
\begin{array}{lll}
g^{(*)}_{ij}(z,t,\theta)=g^K_{ij}(z)+\mathcal{O}(|x|), 
	& g^{(*)}_{\lambda\mu}(z,t,\theta)=g^{(\theta)}_{\lambda\mu}(\theta)+\mathcal{O}(|x|),
	& g_{tt}^{(*)}(z,t,\theta)=1+\mathcal{O}(|x|)\\
g^{(*)}_{i\lambda}(z,t,\theta)=\mathcal{O}(|x|), 
	& g^{(*)}_{it}(z,t,\theta)=\mathcal{O}(|x|),
	& g^{(*)}_{i\lambda}(z,t,\theta)=\mathcal{O}(|x|)
\end{array}
\]
where $g^{(\theta)}_{\lambda\mu}$ denotes a component of the standard round metric
on $S^{m-1}$ in the spherical coordinates $(\theta^1,\dots,\theta^{m-1})$.

Using the same cutoff functions $\xi$ 
and $\eta$ we introduced in the case of interior embeddings, define the function
$u_\varepsilon$ as in section 2.1. For each $\varepsilon\in(0,\frac12)$, set
\[
\begin{array}{rr}
g_\varepsilon(z,t,\theta)=&(\xi g_{ij}^{(1)}+(1-\xi)g_{ij}^{(2)})dz^i dz^j+
	u_\varepsilon^{\frac{4}{n-2}}\Big{(} (\xi g^{(1)}_{tt} +(1-\xi)g^{(2)}_{tt})dt^2\\
{}&+(\xi g_{\lambda\mu}^{(1)}+
	(1-\xi)g_{\lambda\mu}^{(2)})d\theta^\lambda d\theta^\mu+ (\xi g_{t\lambda}^{(1)}+
	(1-\xi)g_{t\lambda}^{(2)})dtd\theta^\lambda]\Big{)}\\
{}&+(\xi g_{it}^{(1)}+(1-\xi)g_{it}^{(2)})dz^idt+(\xi g_{i\lambda}^{(1)}+
	(1-\xi)g_{i\lambda}^{(2)})dz^i d\theta^\lambda.
\end{array}
\]
This defines a metric $g_\varepsilon$ on the tubular annuli 
$V_*^1\setminus\overline{V_*^{\varepsilon^2}}$ for $*=1,2$. 
We set $g_\varepsilon=g_*$ on $M_*\setminus \overline{V_*^1}$.
This gives well-defined metric $g_\varepsilon$ on the disjoint union
$(M_1\setminus V_*^{\varepsilon^2}) \sqcup(M_2\setminus V_*^{\varepsilon^2})$.

Let $\Phi:\mathcal{N}_1(K)\to\mathcal{N}_2(K)$ be the isomorphism of the
normal bundles given in the hypothesis of Theorem $1_c$. 
For each $\varepsilon\in(0,\frac12)$, consider mapping 
$\Psi_\varepsilon$ given by
\begin{align}
\Psi_\varepsilon&:\left(\mathcal{N}_1(K)\setminus\{0\}\right)\sqcup\left(\mathcal{N}_2(K)\setminus\{0\}\right)\to
\left(\mathcal{N}_1(K)\setminus\{0\}\right)\sqcup\left(\mathcal{N}_2(K)\setminus\{0\}\right)\notag\\
\Psi_\varepsilon&(z,t,\theta):=
\begin{cases}
\Phi(z,-t,\theta)&\text{ if }(z,t,\theta)\in\mathcal{N}_1(K)\\
\Phi^{-1}(z,-t,\theta)&\text{ if }(z,t,\theta)\in\mathcal{N}_2(K).
\end{cases}\notag
\end{align}
For each $\varepsilon\in(0,\frac12)$, we construct the generalized connected sum
\[
M=\left( (M_1\setminus V_1^{\varepsilon^2}){\sqcup} 
	(M_2\setminus V_2^{\varepsilon^2})\right)/\sim_\varepsilon
\]
where we introduce a relation $\sim_\varepsilon$ on the annuli
$(V_1^1\setminus V_1^{\varepsilon^2})\sqcup(V_2^1\setminus V_2^{\varepsilon^2})$: 
If $y\in V_1^1\setminus \overline{V_1^{\varepsilon^2}}$, then
$y\sim_\varepsilon (F_2\circ\Psi_\varepsilon\circ F_1^{-1})(y)$. 
Observing that $g_\varepsilon$ is invariant under $\Psi_\varepsilon$, the 
metric descends to $M$. This finishes 
the definition of the family of Riemannian manifolds $(M,g_\varepsilon)$.

Recalling that we assume the mean curvature $H_{g_K}$ vanishes on $\partial K$, the proof of
the following proposition is very similar to argument in Proposition $1_b$ and so we omit it.
\begin{prop1C}
There is a constant $C>0$, independent of $\varepsilon$, such that
\[
|R_{g_\varepsilon}|\leq C\varepsilon^{-1}\cosh^{1-m}(t), \quad |H_{g_\varepsilon}|\leq C\cosh^{2-m}(t)
\]
on $T^\varepsilon(0,0)$ and
\[
\int_M|R_{g_\varepsilon}|d\mu_{g_\varepsilon}=\mathcal{O}(\varepsilon^{m-2}),\quad 
	\int_{\partial M}|H_{g_\varepsilon}| d\sigma_{g_\varepsilon}=\mathcal{O}(\varepsilon^{m-2}).
\]
\end{prop1C}
As for the local a priori estimate, we will need to again decompose the boundary of 
$\partial T^\varepsilon(\alpha_1,\alpha_2)$ into two pieces
\begin{align}
\partial_1 T^\varepsilon(\alpha_1,\alpha_2)=&
	\{(z,t,\theta)\in T^\varepsilon(\alpha_1,\alpha_2)\colon t=\log\varepsilon+\alpha_1
	\text{ or }t=-\log\varepsilon-\alpha_2\}\notag\\
\partial_2 T^\varepsilon(\alpha_1,\alpha_2)=&
	\{(z,t,\theta)\in T^\varepsilon(\alpha_1,\alpha_2)\colon z\in\partial K\}\notag.
\end{align}
We will use the same notation for $\partial_1T^\varepsilon(\alpha_1,\alpha_2)$
and $\partial_2T^\varepsilon(\alpha_1,\alpha_2)$ as we did in the case of boundary embeddings.
There is also an analogue of the estimates in Propositions $2_a$ and $2_b$ for the present case of relative 
embeddings. Its proof is very similar to that of Proposition $2_b$ and we leave it to the reader.
\begin{prop2C}
Given $\gamma\in(0,m-2)$ there are $\varepsilon$-uniform constants 
$\alpha_1,\alpha_2>1$ and $C>0$ satisfying the following statement
for all $\varepsilon\in(0,e^{-\max\{\alpha_1,\alpha_2\}})$. 
If $v,f\in \mathcal{C}^0(T^\varepsilon(\alpha_1,\alpha_2))$ satisfy $\Delta_{g_\varepsilon}v=f$, then
\[
v\leq 
C\psi_\varepsilon^{-\gamma}\left(\sup_{T^\varepsilon(\alpha_1,\alpha_2)}|\psi^{\gamma+2}_\varepsilon f|+
\sup_{\partial_1 T^\varepsilon(\alpha_1,\alpha_2)}|\psi^\gamma_\varepsilon v|+
\sup_{\partial_2 T^\varepsilon(\alpha_1,\alpha_2)}|\psi^{\gamma+1}_\varepsilon\partial_\nu v|\right)
\]
pointwise on $T^\varepsilon(\alpha_1,\alpha_2)$ and 
\[
||v||_{\mathcal{C}^0_\gamma(T^\varepsilon(\alpha_1,\alpha_2))}\leq 
C\left(||f||_{\mathcal{C}^0_{\gamma+2}(T^\varepsilon(\alpha_1,\alpha_2))}+
||v||_{\mathcal{C}^0_\gamma(\partial_1 T^\varepsilon(\alpha_1,\alpha_2))}+
||\partial_\nu v||_{\mathcal{C}^0_{\gamma+1}(\partial_2 T^\varepsilon(\alpha_1,\alpha_2))}\right).
\]
\end{prop2C}

\section{The linear analysis}
Now that we have constructed the generalized connected sum $(M,g_\varepsilon)$,
we will turn our attention to equation (\ref{eq:RYP2}).
At this point, there is no need to consider the interior, boundary, and relative
embedding cases independently as we did in Section 2. Unless otherwise mentioned,
{\bf from now on we will speak of all three cases simultaneously.}  

Our first task will be to study 
the family of linear operators $(\Delta_{g_\varepsilon},\partial_\nu)$ for
$\varepsilon\in(0,\frac12)$. Before we continue, now is a good time to make 
some informal remarks. The first non-zero Steklov eigenvalue of 
$(\Delta_{g_\varepsilon},\partial_\nu)$, 
which we write as $\lambda_{\varepsilon}$, is the smallest number 
such that the following equation admits a non-constant solution $f$
\[
\begin{cases}
\Delta_{g_\varepsilon}f=0&\text{ on } M\\
\partial_\nu f=\lambda_{\varepsilon}f&\text{ on }\partial M.
\end{cases}
\]
In general, $\lambda_{\varepsilon}\to0$ as $\varepsilon\to0$.
For this reason, there is no general result which would provide us a useful 
$\varepsilon$-uniform $\mathcal{C}^0(M)$ 
estimate for our linear problem.

This in mind, we take two measures to combat this degeneracy.
In addition to working in the weighted Banach spaces $\mathcal{C}^0_\gamma(M)$ we
introduced in Section 2, 
we will initially solve (with estimates) a modification 
of the linear problem. Speaking informally, this auxiliary problem is formulated 
by projecting the linear problem along a hand-made model for the first 
non-constant eigenfunction.
This model is a function denoted by $\beta_\varepsilon$ which takes the values
1 on $M_1\setminus V_1^{\varepsilon}$, $-1$ on $M_2\setminus V_2^{\varepsilon}$, and
interpolates between them on the neck so that 
$\int_M\beta_\varepsilon d\mu_{g_\varepsilon}=0$ (see Section 3.1).

Given $\gamma\in(0,m-2)$ and suitable functions $f\in\mathcal{C}^0_{\gamma+2}(M)$,
$\ell\in\mathcal{C}^0_\gamma(\partial M)$,
we will produce a function $u\in\mathcal{C}^0_\gamma(M)$ satisfying
\begin{equation}\label{eq:lPDE}
\begin{cases}
\Delta_{g_\varepsilon}u=f&\text{ on } M\\
\partial_\nu u=\ell-\lambda\beta_\varepsilon&\text{ on }\partial M
\end{cases}
\end{equation}
where $\lambda$ is a real number depending on $f$ and $\ell$.
Notice that the functions $f,\ell$ must satisfy 
\begin{equation}\label{eq:green}
\int_Mfd\mu_{g_\varepsilon}=\int_{\partial M}\ell d\sigma_{g_\varepsilon},
\end{equation} 
which is simply Green's formula applied to $u$. 
We will refer to (\ref{eq:green}) as the orthogonality condition
of equation (\ref{eq:lPDE}).
As we produce this solution, we also obtain an $\varepsilon$-uniform 
$\mathcal{C}^0_\gamma$-norm 
a priori estimate for $u$ using standard elliptic estimates on $(M_*,g_*)$ 
with the local a priori estimate of Propositions $2_a$, $2_b$, and $2_c$.  

Before we begin, it will be useful to state 
a regularity result we will require later in the present section.
The following theorem is a version of elliptic 
$L^p$ estimate, tailored to the Neumann problem.
\begin{theorem}{\rm cf. \cite[Theorem 3.2]{uhl}}
Let $(N,g_N)$ be a compact Riemannian manifold with boundary $\partial N$.
Assume that  $v\in W^{k+2,p}(N,g_N)$ for some $k,p\in\mathbb{N}_0$
satisfies $\int_N v \,d\mu_{g_N}=0$. 
Then there is a constant $C$ depending only on the geometry of $(N,g_N)$, 
$k$, and $p$ such that
\begin{equation}\label{eq:uhl}
||v||_{W^{k+2,p}(N,g_N)}\leq C\left(||\Delta_{g_N}v||_{W^{k,p}(N,g_N)}+||\partial_\nu v||_{W^{k+1,p}_\partial(N,g_N)}\right).
\end{equation}
where the norm $||\cdot||_{W^{k,p}_\partial(N,g_N)}$ is defined by
\[
||F||_{W^{k,p}_\partial(N,g_N)}:=\inf\{||G||_{W^{k,p}(N,g_N)}\colon G\in W^{k,p}(N,g_N), G|_{\partial N}=F\}.
\]
\end{theorem}

\subsection{The linear problem I}
For each $\alpha_1,\alpha_2>1$, let us fix $\rho_1$ and $\rho_2$, two smooth functions 
on $M_1\sqcup M_2$ satisfying
\[
\rho_1=\begin{cases} 1&\text{ on } M_1\setminus T^\varepsilon(\alpha_1,0)\\
0&\text{ on } M_2\setminus T^\varepsilon(0,-2\log \varepsilon-\alpha_1-1)
\end{cases}
\]
\[
\rho_2=\begin{cases} 1&\text{ on } M_2\setminus T^\varepsilon(0,\alpha_2)\\
0&\text{ on } M_1\setminus T^\varepsilon(-2\log \varepsilon-\alpha_2-1,0)
\end{cases}
\]
and $\partial_\nu\rho_1\equiv0$, and $\partial_\nu\rho_2\equiv0$ on $\partial M_1\sqcup\partial M_2$.
Understanding that $\rho_1$ and $\rho_2$ descend to the connected sum $M$, 
we then define $\beta_\varepsilon:M\to\mathbb{R}$ by 
$\beta_\varepsilon:=\rho_1-\rho_2$.

In the case of interior embeddings, where we have not altered the original 
metrics on the boundary, it is immediate that
\[
\int_{\partial M} \beta_\varepsilon d\sigma_{g_\varepsilon}=0
\]
since we assume $\mathrm{Vol}_{g_1}(\partial M_1)=\mathrm{Vol}_{g_2}(\partial M_2)$.
To arrange for $\beta_\varepsilon$ to have vanishing 
average value on the boundary in 
the case of boundary and relative embeddings
(where $d\sigma_{g_\varepsilon}$ is affected by the gluing), we may have to 
choose $\alpha_1$ and $\alpha_2$ differently. However, notice that this can 
always be achieved by only increasing either $\alpha_1$ or $\alpha_2$.
Since the estimate of Lemma 2 also holds for these larger parameters, 
{\bf from now on we will assume that $\alpha_1$ and $\alpha_2$ have been chosen so that 
Propositions $2_a, 2_b,$ and $2_c$ apply and 
$\int_{\partial M}\beta_\varepsilon d\sigma_{g_\varepsilon}=0$.}

In this section we build an approximate solution to (\ref{eq:lPDE})
which is straight-forward to estimate, but accumulates many error terms in 
a gluing process.  This construction is summarized in the 
following lemma which will subsequently
be applied iteratively to establish a genuine solution to the linear 
problem (\ref{eq:lPDE}), with estimates. 

\begin{lemma}\label{lem:linear1}
Let $\gamma\in(0,m-2)$ and $B\in(0,1)$.  There is an $\varepsilon_0>0$ such that
the following statement is satisfied for all $\varepsilon\in(0,\varepsilon_0)$:
Suppose $f\in \mathcal{C}^0_{\gamma+2}(M)$ and 
$\ell\in\mathcal{C}^0_{\gamma+1}(\partial M)$ satisfy
\[
\int_{M}fd\mu_{g_\varepsilon}=\int_{\partial M}\ell d\sigma_\varepsilon.
\]
Then there is $\lambda\in\mathbb{R}$,
a function $u\in\mathcal{C}^0_{\gamma}(M)$, and an error term 
$E\in\mathcal{C}^0_{\gamma+2}(M)$ satisfying
\[
\begin{cases}
\Delta_{g_\varepsilon}u=f +E&\text{ in } M\\
\partial_\nu u=\ell-\lambda\beta_\varepsilon&\text{ on }\partial M\\
\int_M ud\mu_{g_\varepsilon}=0
\end{cases}
\]
Moreover, $u,\lambda,$ and $E$ satisfy the following estimates
\[
\begin{array}{rcl}
||u||_{\mathcal{C}^0_\gamma(M)}&\leq &C(||f||_{\mathcal{C}^0_{\gamma+2}(M)}+||\ell||_{\mathcal{C}_{\gamma+1}^0(\partial M)})\\
|\lambda|&\leq&C(||f||_{\mathcal{C}^0_{\gamma+2}(M)}+||\ell||_{\mathcal{C}_{\gamma+1}^0(\partial M)})\\
||E||_{\mathcal{C}^0_{\gamma+2}(M)}&\leq&C\varepsilon^{B\gamma}(||f||_{\mathcal{C}^0_{\gamma+2}(M)}+||\ell||_{\mathcal{C}_{\gamma+1}^0(\partial M)})
\end{array}
\]
where the constant $C>0$ is independent of $\varepsilon$ and $B$.
\end{lemma}
\begin{proof}
First we let $\rho_T:=1-\rho_1-\rho_2$ so that
$\{\rho_1,\rho_T,\rho_2\}$ forms a partition of unity on $M$. We
decompose $f$ and $\ell$ with respect to this partition, writting
\[
f_1=f\rho_1, \quad f_T=f\rho_T, \quad f_2=f\rho_2,
\]
\[
\ell_1=\ell\rho_1, \quad \ell_T=\ell\rho_T, \quad \ell_2=\ell\rho_2.
\]
Next, we produce an approximate solution on the neck 
$T^\varepsilon(\alpha_1,\alpha_2)$.

\begin{claim}
For the parameters $\gamma, B$ and functions  $f, \ell$ in Lemma 2, there is a unique function 
$\tilde{u}_T\in\mathcal{C}^0_\gamma(T^\varepsilon(\alpha_1,\alpha_2))$ 
satisfying
\begin{equation}\label{eq:neck1}
\begin{cases}
\Delta_{g_\varepsilon}\tilde{u}_T=f_T&\text{ in } T^\varepsilon(\alpha_1,\alpha_2)\\
\tilde{u}_T=0&\text{ on } \partial_1 T^\varepsilon(\alpha_1,\alpha_2)\\
\partial_\nu\tilde{u}_T=\ell_T&\text{ on }\partial_2 T^\varepsilon(\alpha_1,\alpha_2).
\end{cases}
\end{equation}
Moreover, there is a constant $C_T>0$, independent of $\varepsilon$, such that
\[
||\tilde{u}_T||_{\mathcal{C}^0_\gamma(T^\varepsilon(\alpha_1,\alpha_2))}\leq
C_T\left(||f_T||_{\mathcal{C}^0_{\gamma+2}(T^\varepsilon(\alpha_1,\alpha_2))}+||\ell_T||_{\mathcal{C}^0_{\gamma+1}(\partial_2T^\varepsilon(\alpha_1,\alpha_2))}\right).
\]
\end{claim}
\begin{proof}
Notice that $T^\varepsilon(\alpha_1,\alpha_2)$ is a compact manifold with corners. 
This allows us to apply the regularity theory in \cite{L} -- by \cite[Theorem 1]{L}, there is a unique function
\[
\tilde{u}_T\in\mathcal{C}^2\left(T^\varepsilon(\alpha_1,\alpha_2)\cup
\partial_2T^\varepsilon(\alpha_1,\alpha_2)\right)\cap\mathcal{C}^0(\overline{T^\varepsilon(\alpha_1,\alpha_2)})
\]
solving equation (\ref{eq:neck1}).
We may then apply Proposition $2_a, 2_b,$ or $2_c$ with the
parameter $\gamma$ from the hypothesis
of Lemma 2 and the function $\tilde{u}_T$ to arrive at the estimates in the claim.
\end{proof}

We extend the domain of $\tilde{u}_T$ to all of $M$, which we will continue to call $\tilde{u}_T$, 
by declaring $\tilde{u}_T=0$ on $M\setminus T^\varepsilon(\alpha_1,\alpha_2)$.
While $\tilde{u}_T$ may not be differentiable on $\partial_1 T^\varepsilon(\alpha_1,\alpha_2)$, the function
$u_T:=\rho_T\tilde{u}_T$ is differentiable since the support of $\rho_T$ is contained in 
$T^\varepsilon(\alpha_1+1,\alpha_2+1)$.  
One can compute
\begin{align}
\Delta_{g_\varepsilon}u_T&=f_T-q_1-q_2\notag\\
\partial_\nu u_T&=\ell_T-q_1^\partial-q_2^\partial\notag
\end{align}
where $q_*:=\Delta_{g_\varepsilon}(\rho_*\tilde{u}_T)$ and $q_*^\partial:=\partial_\nu(\rho_*\tilde{u}_T)$.  
The quantities $q_*$ and $q_*^\partial$ will be accounted for 
in the next step.

We now turn to the pieces of $M$ which come from the original manifolds $M_*$.
We define $\lambda$ according to the formula
\begin{equation}\label{eq:lambda}
\lambda:=\frac{1}{\int_{\partial M}(\rho_1+\rho_2)d\sigma_{g_\varepsilon}}
	 \left( \int_{\partial M}(\ell\beta_\varepsilon+q_1^\partial-q_2^\partial)  d\sigma_{g_\varepsilon}
	-\int_M(f\beta_\varepsilon +q_1-q_2)d\mu_{g_\varepsilon}\right),
\end{equation}
which can be interpreted as the projection of $f$ and $\ell$ along $\beta_\varepsilon$.
Observe that, for $*=1,2$, this choice of $\lambda$ implies
\begin{equation}\label{eq:Hi}
\int_M(f_*+q_*)d\mu_{g_\varepsilon}-
	\int_{\partial M}(\ell \rho_*+q_*^\partial+(-1)^*\lambda\rho_*)d\sigma_{g_\varepsilon}=0,
\end{equation}  
which we will use later.

Using standard elliptic techniques \cite{A}\cite{uhl}, we may consider a 
distributional solution $\tilde{u}_*$ to the following system
\[
\begin{cases}
\Delta_{g_*}\tilde{u}_*=f_*+q_*+b_*\delta_{\iota_*}&\text{ in }M_*\\
\partial_\nu\tilde{u}_*=\ell_*+q_*^\partial+(-1)^*\lambda\rho_*&\text{ on }\partial M_*\\
\int_M\tilde{u}_*d\mu_{g_*}=0&{}
\end{cases}
\]
where $\delta_{\iota_*}$ denotes the Dirac distribution supported on the submanifold 
$\iota_*(K)$. Applying Green's theorem to $\tilde{u}_*$, the constant $b_*$ is forced to be
\[
b_*=\frac{1}{\mathrm{Vol}_{g_K}(K)}\left(\int_{\partial M_*}(\ell_*+q_*^\partial+(-1)^*\lambda\rho_*)d\sigma_{g_*}
 -\int_{M_*}(f_*+q_*)d\mu_{g_*}\right).
\]
\begin{claim}
There is a constant $C'>0$ independent of $\varepsilon$ such that
\[
|\tilde{u}_*|\leq C'(||f||_{\mathcal{C}^0(M)}+||\ell||_{\mathcal{C}^0(\partial M)})
\]
on $M_*\setminus V^1_*$,
\[
|\tilde{u}_*|\leq C'|x|^{2-m}(||f||_{\mathcal{C}^0_{\gamma+2}(M)}+||\ell||_{\mathcal{C}^0_{\gamma+1}(\partial M)})
\]
on $V^1_*$, and 
\[
|\lambda| \leq C'(||f||_{\mathcal{C}^0_{\gamma+2}(M)}+||\ell||_{\mathcal{C}_{\gamma+1}^0}).
\]
\end{claim}
\begin{proof}
To estimate $\tilde{u}_*$, it will be useful to 
consider the decomposition $\tilde{u}_*=\overline{u}_*+\hat{u}_*$ where
\[
\begin{cases}
\Delta_{g_*}\overline{u}_*=f_*+q_*+\mathrm{Vol}_{g_K}(K)b_*&\text{ in }M_*\\
\partial_\nu\overline{u}_*=\ell_*+q_*^\partial+(-1)^*\lambda\rho_*&\text{ on }\partial M_*\\
\int_{M_*}\overline{u}_*d\mu_{g_*}=0&{}
\end{cases}
\]
\[
\begin{cases}
\Delta_{g_*}\hat{u}_*=-\mathrm{Vol}_{g_K}(K)b_*
+b_*\delta_{\iota_*}&\text{ in }M_*\\
\partial_\nu\hat{u}_*=0&\text{ on }\partial M_*\\
\int_{M_*}\hat{u}_*d\mu_{g_*}=0&{}
\end{cases}
\]
One can think of $\overline{u}_*$ and $\hat{u}_*$ as the finite and Green's function parts 
of $\tilde{u}_*$, respectively.  
Near the submanifold $\iota_*(K)$, one can use the Green's function
construction presented in \cite{A} to see that $\hat{u}_*$ takes the form
\[
\hat{u}_*=\frac{b_*}{(m-2)\omega_{m-1}}\left(|x|^{2-m}+\mathcal{O}(|x|^{3-m})\right)
\]
where $\omega_{m-1}$ is the volume of unit sphere $S^{m-1}$ and the term
$\mathcal{O}(|x|^{3-m})$ depends only on the geometry of $(M_*,g_*)$.
It follows that there is a constant $C_0$, independent of $\varepsilon$,
such that
\begin{equation}\label{eq:uhat}
|\hat{u}_*|\leq C_0 b_*|x|^{m-2}
\end{equation}
on $V^1_*$.

Next, we consider $\overline{u}_*$.
By taking $p=n$ and $k=0$ in the $L^p$ estimate (\ref{eq:uhl}) applied to 
$\overline{u}_*$, there is a constant $C_1>0$ so that
\begin{align}
||\overline{u}_*||_{W^{2,n}(M_*,g_*)}\leq &
C_1\Big{(}\left\lVert f_*+q_*-\frac{Vol_{g_K}(K)}{Vol_{g_*}(M_*) }b_*
\right\rVert_{L^n(M_*,g_*)}\notag\\
{}&\quad\quad\quad\quad+||\ell_*+q_*^\partial+(-1)^*\lambda\rho_*||_{W^{1,n}_\partial(M_*,g_*)}\Big{)}\notag
\end{align}
for $*=1,2$ where $C_1$ depends only on $n$ and the geometry of 
$(M_1,g_1), (M_2,g_2)$.
Now we may use the Sobolev Embedding Theorem \cite[Theorem 2.30]{A} and 
the Trace Theorem \cite[Theorem B.10]{uhl} to obtain the following $\mathcal{C}^0$ estimate
\begin{align}\label{eq:overlineu}
||\overline{u}_*||_{\mathcal{C}^0(M_*)}\leq&
C_2\Big{(}\left\lVert f_*+q_*-\frac{Vol_{g_K}(K)}{Vol_{g_*}(M_*) }b_*
	\right\rVert_{\mathcal{C}^0(M_*)}\notag\\
{}&\quad\quad\quad\quad+||\ell_*+q_*^\partial+(-1)^*\lambda\rho_*||_{\mathcal{C}
	^0(\partial M_*)}\Big{)}
\end{align}
where $C_2$ is a constant depending only on $n$ and the geometry of 
$(M_1,g_1), (M_2,g_2)$.

To finish the proof of the claim, it suffices to estimate $b_*,q_*,$ and $q_*^\partial$. 
It will be convenient to consider the cases $*=1,2$
separately -- in what follows, the statements will be made for $*=1$, though analogous 
arguments hold for $*=2$ and this is left to the reader.
Subtracting (\ref{eq:Hi}) from $b_1$ shows
\begin{align}\label{eq:b1}
b_1=&\frac{1}{\mathrm{Vol}_{g_K}(K)}\Big{(}\int_{\partial_2T^\varepsilon(0,0)
	\setminus\partial_2 T^\varepsilon(\alpha_1,0)}(\ell_1+q^\partial_1-
	\lambda\rho_1)\left(\frac{\sqrt{g^\partial_1}-\sqrt{g^\partial_\varepsilon}}{
	\sqrt{g^\partial_1}}\right)d\sigma_{g_1}\notag\\
{}&\quad\quad\quad\quad-\int_{T^\varepsilon(0,0)\setminus T^\varepsilon(\alpha_1,0)}
	(f_1+q_1)\left(\frac{\sqrt{g_1}-\sqrt{g_\varepsilon}}{\sqrt{g_1}}\right)
	d\mu_{g_1}\Big{)}
\end{align}
where $\sqrt{g_1^\partial}$ and $\sqrt{g^\partial_\varepsilon}$ denote the Riemannian
measures of $g_1|_{\partial M_1}$ and $g_\varepsilon|_{M_1}$, respectively.
Notice that we only integrate over
$T^\varepsilon(0,0)\setminus T^\varepsilon(\alpha_1,0)$
since it contains the supports 
$\mathrm{spt}(\rho_1)\cap\mathrm{spt}(\sqrt{g_1}-\sqrt{g_\varepsilon})$.
We will inspect each term in the expression (\ref{eq:b1}).

On $T^\varepsilon(0,0)\setminus T^\varepsilon(\alpha_1+1,0)$, notice that
$\sqrt{g_1}-\sqrt{g_\varepsilon}=\mathcal{O}(\varepsilon^{m-2})$ and on this portion of
the boundary of $M$ we have 
$\sqrt{g^\partial_1}-\sqrt{g^\partial_\varepsilon}=\mathcal{O}(\varepsilon^{m-2})$. 
Using this, we can find a constant $C_3$ which depends on $\gamma$ and $\alpha_1$,
though not on $\varepsilon$, such that the following inequalities hold
\begin{align}
{}&\int_{\partial_2T^\varepsilon(0,0)
	\setminus\partial_2 T^\varepsilon(\alpha_1,0)}\left\lvert\ell_1
	\left(\frac{\sqrt{g^\partial_1}-\sqrt{g^\partial_\varepsilon}}{
	\sqrt{g^\partial_1}}\right)\right\rvert d\sigma_{g_1}
	\leq C_3 \varepsilon^{m-2}||\ell||_{\mathcal{C}^0_{\gamma+1}(\partial M)}\notag\\
{}&\int_{T^\varepsilon(0,0)\setminus T^\varepsilon(\alpha_1,0)}
	\left\lvert f_1\left(\frac{\sqrt{g_1}-\sqrt{g_\varepsilon}}{\sqrt{g_1}}\right)\right\rvert
	d\mu_{g_1}\leq C_3\varepsilon^{m-2}||f||_{\mathcal{C}^0_{\gamma+2}(M)}\notag.
\end{align}

Next we require pointwise bounds on $q_1$ and $q_1^\partial$ in order to 
estimate (\ref{eq:overlineu}). By definition of $q_1$ and $q_1^\partial$, we have
the expressions
\[
q_1=(\Delta_{g_\varepsilon}\rho_1)\tilde{u}_T+
	2g_\varepsilon(\nabla\rho_1,\nabla\tilde{u}_T)+\rho_1(\Delta_{g_\varepsilon}\tilde{u}_T)\quad\text{ and }\quad
q_1^\partial=\rho_1\partial_\nu\tilde{u}_T
\]
where we have used the fact that $\partial_\nu\rho_1\equiv0$ on $\partial M$.
It is worthwhile to note that the support of $\nabla\rho_1$ satisfies
\[
\mathrm{spt}(\nabla\rho_1)\subset \{y\in M_1\colon e^{-\alpha_1-1}\leq
	\mathrm{dist}_{g_1}(y,\iota_1(K))\leq 1\},
\]
which we emphasize does not depend on $\varepsilon$.  
With this and the pointwise estimates of $g_\varepsilon$ in mind, 
notice that, for any $\alpha_1$ and $\alpha_2$, 
we may assume that $\rho_1$ has been
chosen so that both $|\Delta_{g_\varepsilon}\rho_1|$ 
and $|\nabla\rho_1|_{g_\varepsilon}^2$ are uniformly bounded in $\varepsilon$.
Using this observation and the estimates of Propositions $2_a, 2_b,$ or $2_c$, 
one can show 
\[
||(\Delta_{g_\varepsilon}\rho_1)\tilde{u}_T||_{\mathcal{C}^0_\gamma(M)}\leq 
	C_4(||f||_{\mathcal{C}^0_{\gamma+2}(M)}+
	||\ell||_{\mathcal{C}^0_{\gamma+1}(\partial M)})
\]
for some $C_4$ independent of $\varepsilon$. 
Inspecting (\ref{eq:neck1}), we can find a constant $C_5$, depending on 
$\gamma$ and $\alpha_1$ but not $\varepsilon$, so that 
\[
||\rho_1\Delta_{g_\varepsilon}\tilde{u}_T||_{\mathcal{C}^0(M)}\leq C_5||f||_{\mathcal{C}^0_{\gamma+2}(M)}
\]
\[
||\rho_1\partial_\nu\tilde{u}||_{\mathcal{C}^0(\partial M)}\leq C_5||\ell||_{\mathcal{C}^0_{\gamma+1}(\partial M)}.
\]

The final term we need to estimate is 
$g_\varepsilon(\nabla\rho_1,\nabla\tilde{u}_p)$. 
Let us define 
\[
D_{\alpha_1}:=T^\varepsilon(\alpha_1,0)\setminus T^\varepsilon(\alpha_1+1,0).
\]  
Since $\tilde{u}_p$ is a solution to a Poisson equation 
on the region $D_{\alpha_1}$, we may apply the classical 
gradient estimate \cite{A}, along with the pointwise estimates of $g_\varepsilon$ 
above, to find an $\varepsilon$-uniform constant
$C_6$ satisfying
\[
|\nabla \tilde{u}_T|^2_{g_\varepsilon}(y)\leq 
\frac{C_6}{\mathrm{dist}_{g_1}(y,\partial D_{\alpha_1})}(||\tilde{u}_T||_{\mathcal{C}^0(D_{\alpha_1})}
	+||f_T||_{\mathcal{C}^0(D_{\alpha_1})})
\]
for all $y\in D_{\alpha_1}$. Using this estimate with the Cauchy-Schwarz inequality, 
we can estimate the final term in the expression for $q_1$
\[
|g_\varepsilon(\nabla\rho_1,\nabla\tilde{u}_T)|(y)\leq 
C_7(||\tilde{u}_T||_{\mathcal{C}^0_\gamma(D_{\alpha_1})}+
||f_T||_{\mathcal{C}^0_{\gamma+2}(D_{\alpha_1})})
\]
for another $\varepsilon$-uniform constant $C_7$.

Summarizing our work so far, we have found a constant $C_8$, 
independent of $\varepsilon$, 
such that
\begin{equation}\label{eq:qest}
q_1(y)\leq C_8(||f||_{\mathcal{C}^0_{\gamma+2}(M)}+||\ell||_{\mathcal{C}^0_{\gamma+1}(\partial M)})
\end{equation}
\[
q_1^\partial(y)\leq C_8(||f||_{\mathcal{C}^0_{\gamma+2}(M)}+||\ell||_{\mathcal{C}^0_{\gamma+1}(\partial M})
\]
for all $y\in D_{\alpha_1}$. Notice that 
$C_8$ depends only on the geometry of $(M_1,g_1), (K,g_K)$, $\gamma$, and 
$\alpha_1$.
Integrating (\ref{eq:qest}) yields 
the desired estimate of $\lambda$ from the 
statement of the lemma. 
In turn, this estimate on $\lambda$, (\ref{eq:qest}), and the expression (\ref{eq:b1})
gives an estimate of the form
\[
|b_1|\leq e^{(m-2)t}C_9(||f||_{\mathcal{C}^0_{\gamma+2}(M)}+||\ell||_{\mathcal{C}_{\gamma+1}^0(\partial M)})
\]
Finally, recalling (\ref{eq:uhat}) and (\ref{eq:overlineu}), we have arrived at the desired estimate of $|\tilde{u}_1|$.
\end{proof}

Now we chose cut-off functions which will be used to glue together the functions
$\tilde{u}_1,u_T,$ and $\tilde{u}_2$ from Claims 1 and 2.
For the parameter $B\in(0,1)$ from the hypothesis
of Lemma 2, let $\phi_1,\phi_2:M\to[0,1]$ be smooth functions satisfying
\[
\phi_1=\begin{cases}
1&\text{ on } M_1\setminus T^\varepsilon(-B\log\varepsilon,0)\\
0&\text{ on } M_2\setminus T^\varepsilon(0,-(2-B)\log\varepsilon-1)
\end{cases}
\]
\[
\phi_2=\begin{cases}
1&\text{ on } M_1\setminus T^\varepsilon(0,-B\log\varepsilon)\\
0&\text{ on }M_2\setminus T^\varepsilon(-(2-B)\log\varepsilon-1,0)
\end{cases}
\]
which are monotone in $t$ and have vanishing normal derivatives 
$\partial_\nu\phi_*\equiv0$.
$\phi_1$ and $\phi_2$ are not to be confused with the barrier functions $\phi_\delta$
used in Section 2.2. 
Since $\varepsilon\in(0,e^{-\max(\alpha_1,\alpha_2)})$, we may have 
$\mathrm{spt}(\nabla \phi_*)\subset T^\varepsilon(\alpha_1,\alpha_2)$.
Next, we will define the approximate solution
\[
u:=\phi_1\tilde{u}_1+u_T+\phi_2\tilde{u}_2.
\]
Observe that claims 1 and 2, along with the choice of $\phi_*$, imply the estimate on 
$||u||_{\mathcal{C}^0_{\gamma}(M)}$
in Lemma \ref{lem:linear1}. Our final task will be to inspect the error term.

Since the cut-off functions have vanishing normal derivative, we 
have 
\[
\partial_\nu u=\ell_1+\ell_T+\ell_2=\ell
\]
and so we have accumulated no error term on the boundary.
Moving on the the laplacian of $u$, it is straight-forward to compute 
(keeping the support of 
$\nabla \phi_*$ in mind)
\begin{align}
\Delta_{g_\varepsilon}u=&\Delta_{g_\varepsilon}(\phi_1\tilde{u}_1)+\Delta_{g_\varepsilon}u_p
+\Delta_{g_\varepsilon}(\phi_2\tilde{u}_2)\notag\\
{}=&\Delta_{g_\varepsilon}(\phi_1)\tilde{u}_1+g_\varepsilon(\nabla \phi_1,\nabla \tilde{u}_1)+ \phi_1f\rho_1+
 \phi_1q_1+ \phi_1b_1\delta_{\iota_1(K)}\notag\\
{}&+\Delta_{g_\varepsilon}(\phi_2)\tilde{u}_2+g_\varepsilon(\nabla \phi_2,\nabla \tilde{u}_2)+\phi_2f\rho_2+
\phi_2q_2+\phi_2b_2\delta_{\iota_2(K)}\notag\\
{}&+f\rho_T-q_1-q_2\notag\\
{}=&f+E_1+E_2\notag
\end{align}
where $E_*=(\Delta_{g_\varepsilon}\phi_*)\tilde{u}_*+g_{\varepsilon}(\nabla\phi_*,\nabla\tilde{u}_*)$.
And so the error in the statement of Lemma \ref{lem:linear1} is given by $E:=E_1+E_2$.

By symmetry, it suffices to estimate the term $E_1$.  
Observe that $E_1$ is supported in the annular region 
\[
\{(z,t,\theta)\in T^\varepsilon(0,0):t\in[(1-B)\log\varepsilon,(1-B)\log\varepsilon+1]\}.
\]
By a careful choice of $\phi_1$ and applying the same gradient estimate used 
in the proof 
of Claim 2 (see \cite{A} and \cite{L}), one can find a 
constant $C_{10}$, independent of $\varepsilon$, such that
\[
||E_1||_{\mathcal{C}^0_{\gamma+2}(M)}\leq C_{10}\varepsilon^{B\gamma}(||f||_{\mathcal{C}^0_{\gamma+2}(M)}+
||\ell||_{\mathcal{C}^0_{\gamma+1}(\partial M)}).
\]
This finishes the proof of Lemma \ref{lem:linear1}
\end{proof}

\subsection{The linear problem II}
Lemma \ref{lem:linear1} can be refined by solving (\ref{eq:lPDE}) without accumulating the error term $E$.
\begin{lemma}\label{lem:linear2}
Let $\gamma\in(0,m-2)$.  There exists a choice of parameters $\alpha_1,\alpha_2>1$, $\varepsilon_0>0$,
 and a constant $C>0$ such that the following statement is satisfied for all
$\varepsilon\in(0,\varepsilon_0)$.
Given $f\in \mathcal{C}^0_{\gamma+2}(M)$ and $\ell\in\mathcal{C}^0_{\gamma+1}(\partial M)$ 
satisfying $\int_Mfd\mu_{g_\varepsilon}=\int_{\partial M}\ell d\sigma_{g_\varepsilon}$, there is
a constant $\lambda=\lambda(f,\ell)\in\mathbb{R}$
and a function $u\in\mathcal{C}^0_\gamma(M)$ satisfying
\[
\begin{cases}
\Delta_{g_\varepsilon}u=f &\text{ in } M\\
\partial_\nu u=\ell-\lambda\beta_\varepsilon&\text{ on }\partial M\\
\int_{M}u \ d\mu_{g_\varepsilon}=0
\end{cases}
\]
with the estimates
\begin{align}
||u||_{\mathcal{C}^0_\gamma(M)}&\leq C(||f||_{\mathcal{C}^0_{\gamma+2}(M)}+||\ell||_{\mathcal{C}^0_{\gamma+1}})\notag\\
|\lambda|&\leq C(||f||_{\mathcal{C}^0_{\gamma+2}(M)}+||\ell||_{\mathcal{C}^0_{\gamma+1}})\notag
\end{align}
Moreover, the constant $C>0$ depends only on 
$(M_1,g_1),(M_2,g_2),(K,g_K),\gamma$.
\end{lemma}
\begin{proof}
We will iteratively construct sequences 
\[
f^{(j)}\in\mathcal{C}^0_{\gamma+2}(M), \quad \ell^{(j)}\in\mathcal{C}^0_{\gamma+1}(\partial M), \quad
	u^{(j)}\in\mathcal{C}^0_\gamma,
\]
\[
\lambda^{(j)}\in\mathbb{R},\quad E^{(j)}\in\mathcal{C}^0_{\gamma+2}(M)
\]
and show they converge in appropriate senses.
Setting $f^{(0)}:=f$ and $\ell^{(0)}:=\ell$, Lemma 2 supplies a triple
$u^{(0)},\lambda^{(0)}$, and $E^{(0)}$ solving 
\[
\begin{cases}
\Delta_{g_\varepsilon}u^{(0)}=f^{(0)}+E^{(0)}&\text{ on } M\\
\partial_\nu u^{(0)}=\ell^{(0)}-\lambda^{(0)}\beta_\varepsilon&\text{ on }\partial M
\end{cases}
\]
with estimates. Observe the assumption on $f,\ell$ 
implies that $\int_M E^{(0)}d\mu_{g_\varepsilon}=0$.

Next set $f^{(1)}:=-E^{(0)}$, $\ell^{(1)}:=0$ and again apply Lemma 2
to obtain $u^{(1)},\lambda^{(1)},$
and $E^{(1)}$ satisfying the appropriate equations and estimates.
In general, for $j\geq1$, apply Lemma 2 with $f^{(j)}=-E^{(j-1)}$, $\ell^{(j)}=0$, and 
$B\in(0,1)$ (to be chosen later)
to obtain functions $u^{(j)},\lambda^{(j)},$ and $E^{(j)}$ upon noting that 
$\int_M E^{(j-1)}d\mu_{g_\varepsilon}=0$.
In other words, for each $j\geq1$, we have
\[
\begin{cases}
\Delta_{g_\varepsilon}u^{(j)}=f^{(j)}+E^{(j)}&\text{ in }M\\
\partial_\nu u^{(j)}=-\lambda^{(j)}\beta_\varepsilon&\text{ on }\partial M
\end{cases}
\]
along with a constant $C>0$, independent of $\varepsilon$ and $j$, such that
\[
\begin{array}{rcccl}
||u^{(j)}||_{\mathcal{C}^0_\gamma(M)}&\leq &C||f^{(j)}||_{\mathcal{C}^0_{\gamma+2}(M)}
	&\leq&C(C\varepsilon^{B\gamma})^{j-1}(||f||_{\mathcal{C}^0_{\gamma+2}(M)}+
	||\ell||_{\mathcal{C}^0_{\gamma+1}})\\
|\lambda^{(j)}|&\leq&C||f^{(j)}||_{\mathcal{C}^0_{\gamma+2}(M)}&\leq&
	C(C\varepsilon^{B\gamma})^{j-1}(||f||_{\mathcal{C}^0_{\gamma+2}(M)}+
	||\ell||_{\mathcal{C}^0_{\gamma+1}})\\
||E^{(j)}||_{\mathcal{C}^0_{\gamma+2}(M)}&\leq&
	C\varepsilon^{B\gamma}||f^{(j)}||_{\mathcal{C}^0_{\gamma+2}(M)}
	&\leq&(C\varepsilon^{B\gamma})^j(||f||_{\mathcal{C}^0_{\gamma+2}(M)}+
	||\ell||_{\mathcal{C}^0_{\gamma+1}})
\end{array}
\]

Now consider the partial sums 
\[
v^{(N)}:=\sum_{j=0}^Nu^{(j)},\quad \mu^{(N)}:=\sum_{j=0}^N\lambda^{(j)}
\]
and observe that only one error term remains when computing 
$\Delta_{g_\varepsilon}v^{(N)}$
\[
\begin{cases}
\Delta_{g_\varepsilon}v^{(N)}=f+E^{(N)}&\text{ in } M\\
\partial_\nu v^{(N)}=\ell-\mu^{(N)}\beta_\varepsilon&\text{ on }\partial M.
\end{cases}
\]
Now choose $B\in(0,1)$ so that $C\varepsilon^{B\gamma}$ for all 
$\varepsilon\in(0,\varepsilon_0)$. One can inspect the above estimates from Lemma 2 
and conclude that 
the partial sums $v^{(N)}$, $\mu^{(N)}$ form Cauchy sequences in their respective 
Banach spaces. In fact, the error term vanishes as we take $j\to\infty$
\[
||E^{(N)}||_{\mathcal{C}^0_{\gamma+2}(M)}\leq (C\varepsilon^{B\gamma})^j
(||f||_{\mathcal{C}^0_{\gamma+2}(M)}+||\ell||_{\mathcal{C}^0_{\gamma+1}})\to0.
\]
This gives us a real number $\lambda$ 
and a function $u\in\mathcal{C}^0_{\gamma}$
such that
\[
E^{(N)}\to0,\quad v^{(N)}\to u,\quad \mu^{(N)}\to \lambda,
\]
the convergence being in the appropriate space.
As for the estimates of $u$ and $\lambda$, observe that
\begin{align}
||v^{(N)}||_{\mathcal{C}^0_{\gamma+2}(M)}&\leq\sum_{j=0}^N||u^{(j)}||_{\mathcal{C}^0_{\gamma+2}(M)}\notag\\
{}&\leq\sum_{j=0}^NC(C\varepsilon^{B\gamma}(||f||_{\mathcal{C}^0_{\gamma+2}(M)}+
	||\ell||_{\mathcal{C}^0_{\gamma+1}})\notag\\
{}&\to\frac{C}{1-C\varepsilon^{B\gamma}}(||f||_{\mathcal{C}^0_{\gamma+2}(M)}+
	||\ell||_{\mathcal{C}^0_{\gamma+1}})\notag,
\end{align}
which gives the estimate in Lemma 3. The desired bound on $\lambda$ follows from a 
similar computation.
\end{proof}

\section{The fixed point problem}
The aim of the next two sections is to finish the proofs of Theorems $1_a, 1_b$, and $1_c$
by producing a function $\psi\in C^{\infty}(M)$ which
solves the equation (\ref{eq:RYP2}) on $(M,g_\varepsilon)$ for each $\varepsilon\in(0,\varepsilon_0)$. 
Since we are seeking a small conformal change to $g_\varepsilon$, 
we will write the conformal factor as $\psi=1+u$.  In terms of $u$, equation (\ref{eq:RYP2}) 
becomes
\begin{equation}
\label{eq:thm1}
\begin{cases}
\Delta_{g_\varepsilon}u=F_\varepsilon(u)&\text{ in }M\\
\partial_\nu u=F^\partial_\varepsilon(u)&\text{ on }\partial M
\end{cases}
\end{equation}
where we have introduced the sort-hand notation
\begin{align}
F_\varepsilon(u)&:=c_nR_{g_\varepsilon}(1+u)\notag\\
F^\partial_\varepsilon(u)&:=2c_n(Q(1+u)^{\frac{n}{n-2}}-H_{g_\varepsilon}(1+u))\notag
\end{align}
for some constant $Q$. The convergence statements in 
Theorem 1 will follow as consequences of our construction of $u$.
Upon producing a solution $u$ to (\ref{eq:thm1}), observe that 
$(1+u)^\frac{4}{n-2}g_{\varepsilon}$
will be scalar-flat and have constant boundary mean curvature $Q$.  

In what follows, for a given $\gamma\in(0,m-2)$, we will restrict our attention to 
$u\in\mathcal{C}^0_\gamma(M)$ which lie
in the ball of radius $r_\varepsilon:=\varepsilon^{2\gamma}$ about $0\in\mathcal{C}^0_\gamma(M)$.
We will denote this ball by $B^\gamma_{r_\varepsilon}$.
Let us suppose for a moment that we have in hand a solution $u\in B^\gamma_{r_\varepsilon}$ 
to (\ref{eq:thm1}). Integrating by parts will tell us the mean curvature of 
the resulting conformal metric
\[
Q=\frac{\frac12\int_MR_{g_\varepsilon}(1+u)d\mu_{g_\varepsilon}+
\int_{\partial M}H_{g_\varepsilon}(1+u)
d\sigma_{g_\varepsilon}}{\int_{\partial M}(1+u)^{\frac{n}{n-2}}d\sigma_{g_\varepsilon}}.
\]
Using the $L^1$ estimates on $R_{g_\varepsilon}$ and $H_{g_\varepsilon}$ 
from Propositions
$1_a,1_b,$ and $1_c$, one finds
$|Q|=\mathcal{O}(\varepsilon^{m-2})$.

Before we solve (\ref{eq:thm1}), we will first use our linear analysis to establish a 
solution to the following projected version of the problem
\begin{equation}\label{eq:preyamabe}
\begin{cases}
\Delta_{g_\varepsilon}u= F_\varepsilon(u)&\text{ in }M\\
\partial_\nu u=F^\partial_\varepsilon(u)-\lambda_{F_\varepsilon(u)}\beta_\varepsilon
&\text{ on }\partial M.
\end{cases}
\end{equation}
Later, we will arrange for the vanishing of 
term $\lambda_{F_\varepsilon(u)}$, giving a genuine solution to (\ref{eq:thm1}).

To phrase (\ref{eq:preyamabe}) as a fixed point problem, we introduce the following maps
\begin{align}
F_\varepsilon&:\mathcal{C}^0_\gamma(M)\to
\mathcal{C}^0_{\gamma+2}(M)\times\mathcal{C}_{\gamma+1}^0(\partial M),\quad
v\mapsto (F_\varepsilon(v),F_\varepsilon^\partial(v))\notag\\
G_\varepsilon&:\mathcal{C}^0_{\gamma+2}(M)\times\mathcal{C}_{\gamma+1}^0(\partial M)
\to\mathcal{C}^0_\gamma(M),\quad
(v,w)\mapsto G_\varepsilon(v,w)\notag
\end{align}
where $G_\varepsilon(v,w)$ is the solution to the boundary problem
\[
\begin{cases}
\Delta_{g_\varepsilon}G_\varepsilon(v,w)=v&\text{ in } M\\
\partial_\nu G_\varepsilon(v,w)=
	w-\lambda_{G_\varepsilon(v,w)}\beta_\varepsilon&\text{ on }\partial M,
\end{cases}
\]
whose existence is given by Lemma \ref{lem:linear2}.
Evidently, solving (\ref{eq:preyamabe}) is equivalent to finding 
a fixed point of the composition
\[
P_\varepsilon:\mathcal{C}^0_\gamma(M)\to\mathcal{C}^0_\gamma(M),\quad 
	v\mapsto G_\varepsilon(F_\varepsilon(v),F^\partial_\varepsilon(v))
\]
for some $\gamma$.

\begin{prop3} 
Let $\gamma\in(0,\frac12)$.  There is an $\varepsilon_0>0$ such that 
$P_\varepsilon(B^\gamma_{r_\varepsilon})\subset B^\gamma_{r_\varepsilon}$ for all 
$\varepsilon\in(0,\varepsilon_0)$.
\end{prop3}
\begin{proof}
As usual, $C_k$ for $k=1,2,3\dots$ will denote positive constants
independent of $\varepsilon$.
For $v\in B_{r_\varepsilon}^\gamma$,
we may apply Lemma \ref{lem:linear2} with the functions 
$F_\varepsilon(v),F^\partial_\varepsilon(v))$ 
to get a solution, $P_\varepsilon(v)$, of the linear problem along with the estimate
\[
||P_\varepsilon(v)||_{\mathcal{C}^0_{\gamma}(M)}\leq C_1\left(||F_\varepsilon(v)||_{\mathcal{C}^0_{\gamma+2}(M)}
+||F^\partial_\varepsilon(v)||_{\mathcal{C}_{\gamma+1}^0(\partial M)}\right).
\]
It is suffices to dominate $||F_\varepsilon(v)||_{\mathcal{C}^0_{\gamma+2}(M)}$ and 
$||F^\partial_\varepsilon(v)||_{\mathcal{C}_{\gamma+1}^0(\partial M)}$ by the product of 
$r_\varepsilon$ and some positive power of $\varepsilon$.

We begin with the first summand.  
Applying Propositions $1_a,1_b,1_c$ and the definition of $\psi_\varepsilon$, 
\begin{align}
|F_\varepsilon(v)\psi_\varepsilon^{\gamma+2}|&
	\leq C_2(|R_{g_\varepsilon}|\psi_\varepsilon^{\gamma+2}+
	|R_{g_\varepsilon}|\cdot |v|\psi_\varepsilon^{\gamma+2})\notag\\
{}&\leq C_3(\varepsilon^{m-2}+r_\varepsilon\varepsilon^{m-2})\notag\\
{}&\leq C_4r_\varepsilon\varepsilon^{m-2-2\gamma}.\notag
\end{align}
For the second summand in the estimate, we have
\begin{align}
|F^\partial_\varepsilon(v)|\psi_\varepsilon^{\gamma+1}&\leq C_5(\psi_\varepsilon^{\gamma+1}|Q|(1+v)^{\frac{n}{n-2}}-\psi_\varepsilon^{\gamma+1}|H_{g_{\varepsilon}}|(1+v))\notag\\
{}&\leq C_6\varepsilon^{m-2}r_\varepsilon.\notag
\end{align}
Together, we have shown 
\[
||P_\varepsilon||_{\mathcal{C}^0_\gamma(M)}\leq C_7r_\varepsilon\varepsilon^{m-2-2\gamma},
\]
as claimed. 
\end{proof}

It is a good time to observe a fact we will use later
-- the proofs in this section hold if 
$|Q|$ was only $\mathcal{O}(\varepsilon^{\frac{m-2}{2}})$, 
so long as we restrict ourselves to $\gamma\in(0,\frac14)$. 
Now we are ready to solve (\ref{eq:preyamabe}).
\begin{prop4}
Let $\gamma\in(0,\frac12)$. There exists an $\varepsilon_0>0$
so that, for each  $\varepsilon\in(0,\varepsilon_0)$, (\ref{eq:preyamabe}) has a smooth solution $u\in B_{r_\varepsilon}^\gamma$.
\end{prop4}
\begin{proof}
We will proceed by showing that the mapping $P_\varepsilon$ is 
contractive on the ball 
$B^\gamma_{r_\varepsilon}$. In other words, we will show that there is a 
$\varepsilon_0>0$ 
so that 
\[
||P_\varepsilon(u)-P_\varepsilon(v)||_{\mathcal{C}^0_\gamma(M)}\leq K||u-v||_{\mathcal{C}^0_\gamma(M)}
\]
for all $\varepsilon\in(0,\varepsilon_0)$ and $u,v\in B^\gamma_{r_\varepsilon}$.
We begin by applying Lemma \ref{lem:linear2}
\[
||P_\varepsilon(u)-P_\varepsilon(v)||_{\mathcal{C}^0_\gamma(M)}\leq C
\left(||F_\varepsilon(u)-F_\varepsilon(v)||_{\mathcal{C}^0_{\gamma+2}(M)}+
||F^\partial_\varepsilon(u)-F^\partial_\varepsilon(v)||_{\mathcal{C}_{\gamma+1}^0(\partial M)}\right),
\]
where $C>0$ is independent of $\varepsilon$.
By Proposition 3, all involved terms lie in $B^\gamma_{r_\varepsilon}$ for small $\varepsilon$.

For the first summand, keeping in mind the pointwise estimate on $|R_{g_\varepsilon}|$ from Propositions $1_a, 1_b$, and $1_c$,
and the restriction on $m$, we find
\begin{align}
\psi^{\gamma+2}_\varepsilon|F_\varepsilon(u)-F_\varepsilon(v)|&\leq
	C_8\psi^{\gamma+2}_\varepsilon|R_{g_\varepsilon}(u-v)|\notag\\
{}&\leq C_9\varepsilon\cosh^{3-m}(t)||u-v||_{\mathcal{C}^0_\gamma(M)}\notag\\
{}&\leq C_9\varepsilon||u-v||_{\mathcal{C}^0_\gamma(M)}\notag.
\end{align}
We can perform a similar estimate for the boundary term
\begin{align}
\psi_\varepsilon^{\gamma+1}|F^\partial_\varepsilon(u)-F^\partial_\varepsilon(v)|&
	=C_{10}\varepsilon \cosh(t)\psi_\varepsilon^{\gamma}|Q((1+u)^{\frac{n}{n-2}}-(1+v)^{\frac{n}{n-2}})-H_{g_\varepsilon}(u-v)|\notag\\
{}&\leq C_{11}|Q|\cdot ||u-v||_{\mathcal{C}^0_\gamma(\partial M)}+C_{12}\varepsilon \cosh(t)|H_{g_\varepsilon}|\cdot ||u-v||_{\mathcal{C}^0_\gamma(\partial M)}\notag\\
{}&\leq ||u-v||_{\mathcal{C}^0_\gamma}(C_{13}\varepsilon^{m-2}+C_{14}\varepsilon)\notag.
\end{align}
Since all the constants $C_i$ are independent of $\varepsilon$, we can find an $\varepsilon_0>0$ 
which makes $P_\varepsilon$ a contractive mapping on
$B^\gamma_{r_\varepsilon}$ for $\varepsilon<\varepsilon_0$.

The Banach fixed point theorem applied to $P_\varepsilon$ on $B^\gamma_{r_\varepsilon}$
gives a fixed point of $P_\varepsilon$, which we call $u_\varepsilon$. 
Evidently, $u_\varepsilon$ is a solution to equation (\ref{eq:preyamabe}), concluding 
the proof of Proposition 4.
\end{proof}

\section{Vanishing of $\lambda_{F_\varepsilon(v)}$}
In the last section we found, for all sufficiently small $\varepsilon$, a solution 
$u_\varepsilon\in\mathcal{C}^0_\gamma(M)$ to 
\[
\begin{cases}
\Delta_{g_\varepsilon}u_\varepsilon= F_\varepsilon(u_\varepsilon)&\text{ in }M\\
\partial_\nu u_\varepsilon=F^\partial_\varepsilon(u_\varepsilon)-
	\lambda_{F_\varepsilon(u_\varepsilon)}\beta_\varepsilon
&\text{ on }\partial M.
\end{cases}
\]
The corresponding conformal metric $(1+u_\varepsilon)^{\frac{4}{n-2}}g_\varepsilon$ will be
scalar flat, but will have boundary mean curvature equal to
\[
Q-\frac{1}{2c_n}(1+u)^{\frac{-n}{n-2}}\lambda_{F_\varepsilon(u_\varepsilon)}\beta_\varepsilon
\]
which is non-constant.  Next, we will show that $\varepsilon$-small 
conformal changes can be made to the original metrics
$g_1$ and $g_2$ before applying the gluing procedure such that, 
after applying the above construction and fixed point argument, 
the new projection term $\lambda_{\tilde{F}_\varepsilon(u_\varepsilon)}$ will vanish.

Fix $\tilde{w}_1$ and $\tilde{w}_2$, two non-zero smooth functions 
supported on the interiors of $M_1\setminus V^\varepsilon_1$ and 
$M_2\setminus V^\varepsilon_2$, respectively.  
For real parameters $a_*$ ($*=1,2$) which will be chosen later, we consider the functions 
\[
w_*:=a_*\varepsilon^{\frac{m-2}{2}}\tilde{w}_*
\]
and use them to deform the original metrics
\[
\tilde{g}_*:=(1+w_*)^{\frac{4}{n-2}}g_*.
\]
Replacing $g_1$ and $g_2$ with  $\tilde{g}_1$ and $\tilde{g}_2$ 
in the geometric gluing construction presented in section 3, 
we produce a new family of metrics $\tilde{g}_\varepsilon$ on the 
generalized connected sum $M$.
Of course, $\tilde{g}_\varepsilon$ only differs from $g_\varepsilon$ on the supports of 
$w_1$ and $w_2$.  Keeping in mind that 
$\sup_M|w_*|=\mathcal{O}(\varepsilon^{\frac{n-2}{2}})$, all of the analysis we have done on 
the family of linear operators $(\Delta_{g_\varepsilon},\partial_\nu)$ 
also holds for the new family $(\Delta_{\tilde{g}_\varepsilon},\partial_\nu)$.
Namely, the proof of the a priori estimate in Lemma (\ref{lem:linear2}) also works for the metrics
$\tilde{g}_\varepsilon$.  As usual, we will assume that $\alpha_1$ and $\alpha_2$
have be chosen so that $\int_{\partial M}\beta_\varepsilon d\sigma_{\tilde{g}_\varepsilon}=0$.

Next, we need to gather information about the new scalar curvature and 
boundary mean curvature. Notice that the support of $R_{\tilde{g}_\varepsilon}$ 
has three disjoint components -- 
$T^\varepsilon(0,0)$ and the supports of $w_*$.
Since $R_{\tilde{g}_\varepsilon}$ agrees with $R_{g_\varepsilon}$ on $T^\varepsilon(0,0)$, 
we still have the estimate of Propositions $1_a,1_b$, and $1_c$ there.
On the support of $w_*$, the formula for scalar curvature under conformal change reads
\[
R_{\tilde{g}_\varepsilon}=R_{\tilde{g}_*}=-\frac{1}{c_n}(1+w_*)^{-\frac{n+2}{n-2}}\Delta_{g_*}w_*
\]
and we conclude that $R_{\tilde{g}_\varepsilon}=\mathcal{O}(\varepsilon^{\frac{m-2}{2}})$
on the supports of $w_*$.
Hence, there is a constant $C>0$ such that
\[
|R_{\tilde{g}_\varepsilon}|\leq C\varepsilon^{\frac{m-2}{2}}\psi_\varepsilon^{1-m}(t).
\]
As for the mean curvature of the boundary, $H_{\tilde{g}_\varepsilon}$ does not differ 
from $H_{g_\varepsilon}$ since $w_*$ is supported away from the boundary.

Now, upon restricting our choice of $\gamma$ to the interval $(0,\frac14)$, we may
apply the fixed point argument from Section 4 to produce a solution 
$\tilde{u}_\varepsilon\in B^\gamma_{r_\varepsilon}\subset \mathcal{C}^0_\gamma(M)$ to
\[
\begin{cases}
\Delta_{\tilde{g}_\varepsilon}\tilde{u}_\varepsilon=\tilde{F}_\varepsilon(\tilde{u}_\varepsilon)&
	\text{ in }M\\
\partial_\nu \tilde{u}_\varepsilon=\tilde{F}^\partial_\varepsilon(\tilde{u}_\varepsilon)-
	\lambda_{\tilde{F}_\varepsilon(\tilde{u}_\varepsilon)}\beta_\varepsilon&
	\text{ on }\partial M
\end{cases}
\]
where $\tilde{F}_\varepsilon(u):=c_nR_{\tilde{g}_\varepsilon}(1+u)$ and 
$\tilde{F}^\partial_\varepsilon(u):=2c_n(\widetilde{Q}(1+u)^{\frac{n}{n-2}}
	-H_{\tilde{g}_\varepsilon}(1+u))$.
Once this is achieved, the conformal metric 
$(1+\tilde{u}_\varepsilon)^{\frac{4}{n-2}}\tilde{g}_\varepsilon$
will be scalar flat and have boundary mean curvature equal to
\[
\widetilde{Q}-\frac{1}{2c_n}(1+
\tilde{u}_\varepsilon)^{-\frac{n}{n-2}}\lambda_{\tilde{F}_\varepsilon(\tilde{u}_\varepsilon)}\beta_\varepsilon
\]
where the constant $\widetilde{Q}$ can be computed by integrating by parts 
\[
\widetilde{Q}=\frac{\frac12\int_MR_{\tilde{g}_\varepsilon}(1+\tilde{u}_\varepsilon)d\mu_{\tilde{g}_\varepsilon}+
	\int_{\partial M}H_{\tilde{g}_\varepsilon}(1+\tilde{u}_\varepsilon)d\sigma_{\tilde{g}_\varepsilon}}{
	\int_{\partial M}(1+\tilde{u}_\varepsilon)^{\frac{n}{n-2}}d\sigma_{\tilde{g}_\varepsilon}}.
\]
As before, the projection term $\lambda_{\tilde{F}_\varepsilon(\tilde{v}_\varepsilon)}$
may be non-zero, though it now (continuously) depends on the parameters $a_*$.
We will exploit this to establish the following proposition, 
concluding the proof of Theorems $1_a, 1_b,$ and $1_c$.
The following properties of the metrics $\tilde{g}_*$ will be useful in our 
computations later this section
\begin{align}
\Delta_{\tilde{g}_*}\cdot&=-\frac{1}{2c_n}(1+w_*)^{-\frac{n+2}{n-2}}g_*(\nabla w_*,\nabla \cdot)+
	(1+w_*)^{\frac{-4}{n-2}}\Delta_{g_*} \cdot\notag\\
d\mu_{\tilde{g}_{\varepsilon}}&=(1+w_*)^{\frac{2n}{n-2}}d\mu_{g_\varepsilon}\notag.
\end{align}

\begin{prop5}
For small $\varepsilon$, there is a choice of the real parameters $a_1$ and 
$a_2$ such that the resulting
rough projection $\lambda_{\tilde{F}_\varepsilon(\tilde{u}_\varepsilon)}$ vanishes.
\end{prop5}
\begin{proof}
It suffices to show that the sign of 
$\lambda_{\tilde{F}_\varepsilon(\tilde{u}_\varepsilon)}$ can be changed by manipulating $a_1$ and $a_2$.
From the proof of Lemma \ref{lem:linear2}, we may regard 
$\lambda_{\tilde{F}_\varepsilon(u_\varepsilon)}$ as the following sum
\[
\lambda_{\tilde{F}_\varepsilon(u_\varepsilon)}=
\sum_{j=0}^\infty\lambda^{(j)}_{\tilde{F}_\varepsilon(u_\varepsilon)}
\]
where each term has estimate 
\[
|\lambda^{(j)}_{\tilde{F}_\varepsilon(\tilde{v}_\varepsilon)}|\leq 
C(C\varepsilon^{B\gamma})^j(||\tilde{F}_\varepsilon(\tilde{u}_\varepsilon)||_{\mathcal{C}^0_\gamma(M)}+||\tilde{F}_\varepsilon^\partial(\tilde{u}_\varepsilon)||_{\mathcal{C}^0_{\gamma+1}(\partial M)}),
\]
where $C>0$ is uniform in $\varepsilon$.
From this expression we see that the sign of $\lambda_{\tilde{F}_{\varepsilon}(u_\varepsilon)}$, 
for small $\varepsilon$ and an appropriate choice of $B$, 
is determined by the first term in the sum. We will need to recall 
the formula for $\lambda^{(0)}$ from the proof of Lemma \ref{lem:linear2}
\begin{align}
\lambda^{(0)}:=&\frac{1}{\int_{\partial M}(\rho_1+\rho_2)d\sigma_{\tilde{g}_\varepsilon}}
	\Big(\int_M\tilde{F}_\varepsilon(\tilde{u}_\varepsilon)\beta_\varepsilon d\mu_{\tilde{g}_\varepsilon}-
\int_{\partial M}\tilde{F}^\partial_\varepsilon(\tilde{u}_\varepsilon) \beta_\varepsilon d\sigma_{\tilde{g}_\varepsilon}+\notag\\
{}&+\int_M(\Delta_{\tilde{g}_\varepsilon}(\rho_1\tilde{u}_T)-
	\Delta_{\tilde{g}_\varepsilon}(\rho_2\tilde{u}_T))d\mu_{\tilde{g}_\varepsilon}-
	\int_{\partial M}(\partial_{\tilde{\nu}}(\rho_1\tilde{u}_T)-\partial_{\tilde{\nu}}(\rho_2\tilde{u}_T))d\sigma_{\tilde{g}_\varepsilon}\Big)\notag
\end{align}
where $\tilde{u}_T$ is the solution to 
\[
\begin{cases}
\Delta_{\tilde{g}_\varepsilon}\tilde{u}_T=\tilde{F}_\varepsilon(\tilde{u_\varepsilon})\rho_T
	&\text{ on }T^\varepsilon(\alpha_1,\alpha_2)\\
\tilde{u}_T\equiv0&\text{ on }\partial_1 T^\varepsilon(\alpha_1,\alpha_2)\notag\\
\partial_\nu \tilde{u}_T=\tilde{F}^\partial_\varepsilon(\tilde{u}_\varepsilon)\rho_T&\text{ on }\partial_2T^\varepsilon(\alpha_1,\alpha_2)
\end{cases}
\]
which originally appeared in the first step in the proof of 
Lemma 3. Next, we will 
inspect each of the terms in this expression for $\lambda^{(0)}$.

Unpacking the notations in the first term, we have
\begin{align}
\int_M\tilde{F}_\varepsilon(\tilde{v}_\varepsilon)\beta_\varepsilon d\mu_{\tilde{g}_\varepsilon}&=
	c_n\int_M(R_{\tilde{g}_1}+R_{g_\varepsilon}+R_{\tilde{g}_2})(1+\tilde{u}_\varepsilon)
	(\rho_1-\rho_2)d\mu_{\tilde{g}_\varepsilon}\notag.
\end{align}
Recalling that $\tilde{u}_\varepsilon$ lies in 
$B^\gamma_{r_\varepsilon}\subset \mathcal{C}^0_\gamma(M)$
and applying the pointwise estimate of $R_{g_\varepsilon}$, it is straightforward to show
\[
\int_MR_{g_\varepsilon}(1+\tilde{u}_\varepsilon)\rho_*d\mu_{\tilde{g}_\varepsilon}=
	-4m\mathrm{Vol}(K)\omega_{m-1}+\mathcal{O}(e^{-\alpha_*}\varepsilon^{m-2})
\]
and
\[
\int_MR_{\tilde{g}_*}\rho_*d\mu_{\tilde{g}_\varepsilon}=\frac{1}{c_n}\int_{M_*}|\nabla w_*|_{g_*}^2d\mu_{g_*}
\]
where $\omega_{m-1}$ denotes the volume of the unit $(m-1)$-sphere. 

After integrating by parts, the remaining piece of the first term can be written as
\[
\int_MR_{\tilde{g}_*}\tilde{u}_\varepsilon d\mu_{\tilde{g}_\varepsilon}=
	\int_{M_*}w_*\Delta_{g_*}\tilde{u}_\varepsilon d\mu_{g_*}+\int_{\partial M_*}w_*\partial_\nu\tilde{u}_\varepsilon d\sigma_{g_*}+\mathcal{O}(\varepsilon^{m-2+\gamma}).
\]
Now we Taylor expand and rearrange the above expression for $\Delta_{\tilde{g}_*}$ 
and $\partial_{\tilde{\nu}}$
\begin{align}
\Delta_{g_*}\tilde{u}_\varepsilon=&\left(1+\frac{4}{n-2} w_*+\mathcal{O}(\varepsilon^{m-2})\right)
\Delta_{\tilde{g}_*}\tilde{u}_\varepsilon
-2g_*(\nabla w_*,\nabla \tilde{u}_\varepsilon)+\notag\\
{}&+2w_* g_*(\nabla w_*,\nabla\tilde{u}_\varepsilon)+\mathcal{O}(\varepsilon^{m-2+2\gamma})\notag\\
\partial_{\nu}\tilde{u}_\varepsilon=&\left(1+\frac{2}{n-2}w_*+\mathcal{O}(\varepsilon^{m-2})\right)
\partial_{\tilde{\nu}}\tilde{u}_\varepsilon\notag
\end{align}
and multiply by $w_*$ to find
\begin{align}
\int_MR_{\tilde{g}_*}\tilde{u}_\varepsilon d\mu_{\tilde{g}_\varepsilon}&=\int_{M_*}w_*\tilde{F}_\varepsilon(\tilde{u}_\varepsilon)d\mu_{g_*}+\int_{\partial M}w_*(\tilde{F}^\partial_{\varepsilon}(\tilde{u}_\varepsilon)-\lambda_{\tilde{F}_\varepsilon(\tilde{u}_\varepsilon)}\beta_\varepsilon)d\sigma_{g_\varepsilon}+\mathcal{O}(\varepsilon^{m-2+\gamma})\notag\\
	{}&=\int_M|\nabla w_*|_{g_*}^2d\mu_{g_*}-\lambda_{\tilde{F}_\varepsilon(\tilde{u}_\varepsilon)}\mathcal{O}(\varepsilon^{\frac{m-2}{2}})+\mathcal{O}(\varepsilon^{m-2+\gamma})\notag
\end{align}
where we have used the formula for $R_{\tilde{g}_{\varepsilon}}$ in the expression for 
$\tilde{F}_\varepsilon(\tilde{u}_{\varepsilon})$ and integrated by parts. To summarize our efforts so far, we have found
\begin{align}
\int_{M}\tilde{F}_\varepsilon(\tilde{v}_\varepsilon)\beta_\varepsilon d\mu_{\tilde{g}_\varepsilon}&=
(c_n-1)\left(\int_{M_1}|\nabla w_1|_{g_1}^2d\mu_{g_1}-\int_{M_2}|\nabla w_2|_{g_2}^2d\mu_{g_2}\right)-\lambda_{\tilde{F}_\varepsilon(\tilde{u}_\varepsilon)}\mathcal{O}(\varepsilon^{\frac{m-2}{2}})+\notag\\
{}&+\mathcal{O}(e^{-\max(\alpha_1,\alpha_2)}\varepsilon^{m-2}).\label{eq:est1}
\end{align}

Moving along to the next term in the expression for $\lambda^{(0)}$, we have
\[
\int_{\partial M}\tilde{F}^\partial_\varepsilon(\tilde{u}_\varepsilon)\beta_\varepsilon d\sigma_{\tilde{g}_\varepsilon}=
2c_n\int_{\partial M}(\widetilde{Q}(1+\tilde{u}_\varepsilon)^{\frac{n}{n-2}}-H_{\tilde{g}_\varepsilon}(1+\tilde{u}_\varepsilon))(\rho_1-\rho_2)d\sigma_{\tilde{g}_\varepsilon}.
\]
Now since $H_{\tilde{g}_\varepsilon}\equiv H_{g_\varepsilon}$, we have
\[
\int_{\partial M}H_{\tilde{g}_\varepsilon}(1+\tilde{u}_\varepsilon)\rho_*d\sigma_{\tilde{g}_\varepsilon}=\mathcal{O}(e^{-\alpha_*}\varepsilon^{m-2})
\]
which can be seen by computing $H_{g_\varepsilon}$ on this portion of the neck, noting that
the cut off functions $\xi$ and $\eta$ both take the value of 1 on the support of $\rho_1$.

Now is a good time to comment on the convergence statements in the main theorems.
As we have mentioned already, we may apply the pointwise estimate
of $R_{\tilde{g}_\varepsilon}$ and the $\mathcal{C}^0_\gamma$-norm
of $\tilde{v}_\varepsilon$ to find that $\tilde{Q}$ satisfies the estimate
\[
|\widetilde{Q}|=\mathcal{O}(\varepsilon^{\frac{m-2}{2}}).
\]
Evidently, $\tilde{F}_\varepsilon(\tilde{u}_\varepsilon)=\mathcal{O}(\varepsilon^{\frac{m-2}{2}})$
on the support of $w_*$ and 
$\lambda_{\tilde{F}_\varepsilon(\tilde{u}_\varepsilon)}=\mathcal{O}(\varepsilon^{(m-2)/2})$.
Using the computations made in this section, one can inspect the formula for 
$\widetilde{Q}$ and improve our estimate to $|\widetilde{Q}|=\mathcal{O}(\varepsilon^{m-2})$, 
as claimed in Theorems $1_a,1_b,$ and $1_c$.
This can be used to estimate the remaining term in the expression for $\int_{\partial M} \tilde{F}_\varepsilon(\tilde{u}_\varepsilon)\beta_\varepsilon d\sigma_{\tilde{g}_\varepsilon}$ and conclude
\begin{equation}\label{eq:est2}
\int_{\partial M}\tilde{F}^\partial_\varepsilon(\tilde{v}_\varepsilon)\beta_\varepsilon d\sigma_{\tilde{g}_\varepsilon}=\mathcal{O}(e^{-\max(\alpha_1,\alpha_2)}\varepsilon^{m-2}).
\end{equation}

The final two integrals in the expression for $\lambda^{(0)}$ will be treated together.
Integrating by parts, we have
\begin{align}
\int_M\Delta_{g_\varepsilon}(\rho_*\tilde{u}_T)d\mu_{\tilde{g}_\varepsilon}-\int_{\partial M}\partial_\nu(\rho_*\tilde{u}_T)d\sigma_{\tilde{g}_\varepsilon}&=\int_M(\rho_*\Delta_{g_\varepsilon}\tilde{u}_T+2(g_\varepsilon(\nabla\rho_*,\nabla\tilde{u}_T)+\tilde{u}_T\Delta_{g_\varepsilon}\rho_*)-\notag\\
{}&\quad\quad\quad\tilde{u}_T\Delta_{g_\varepsilon}\rho_*)d\mu_{g_\varepsilon}-\int_{\partial M}\rho_*\partial_\nu\tilde{u}_Td\sigma_{g_\varepsilon}\notag\\
{}&=\int_M\rho_*\rho_T\tilde{F}_\varepsilon(\tilde{u}_\varepsilon)-\tilde{u}_T\Delta_{g_\varepsilon}\rho_*d\mu_{g_\varepsilon}-\notag\\
{}&\quad\quad\quad\int_{\partial M}\rho_*\rho_T\tilde{F}^\partial_\varepsilon(\tilde{u}_\varepsilon)d\sigma_{g_\varepsilon}\notag
\end{align}
where we have used the fact that $\partial_\nu\rho_*\equiv0$. In order to proceed, will need the pointwise estimate of Propositions $2_a, 2_b,$ and $2_c$:
\begin{align}
\tilde{u}_T&\leq
	C\psi_\varepsilon^\gamma\left(||\tilde{F}_\varepsilon(\tilde{u}_\varepsilon)||_{\mathcal{C}^0_{\gamma+2}(T^\varepsilon(\alpha_1,\alpha_2))}+||\tilde{F}^\partial_\varepsilon(\tilde{u}_\varepsilon)||_{\mathcal{C}^0_{\gamma+1}(\partial_2 T^\varepsilon(\alpha_1,\alpha_2))}\right)\notag\\
{}&\leq C'\varepsilon^{m-2}\psi^\gamma_\varepsilon\notag
\end{align}
for $C,C'>0$ independent of $\varepsilon$.
Keeping in mind that $\rho_*\rho_T$ and $\Delta_{\tilde{g}_\varepsilon}(\rho_*)$ vanish outside of 
$T^\varepsilon(\alpha_1,-2\log\varepsilon-\alpha_1-1)$ if $*=1$ and $T^\varepsilon(-2\log\varepsilon-\alpha_2-1,\alpha_2)$ if $*=2$, one can use the pointwise estimate on $\tilde{u}_T$ to find
\[
\int_M\Delta_{g_\varepsilon}(\rho_*\tilde{u}_T)d\mu_{\tilde{g}_\varepsilon}-\int_{\partial M}\partial_\nu(\rho_*\tilde{u}_T)d\sigma_{\tilde{g}_\varepsilon}=
 	\mathcal{O}(e^{-\alpha_*}\varepsilon^{m-2}).
\]

Combining the above estimates, we have
\begin{align}
\lambda^{(0)}=&
	(c_n-1)\left(\int_{M_1}|\nabla w_1|_{g_1}^2d\mu_{g_1}-\int_{M_2}|\nabla w_2|_{g_2}^2d\mu_{g_2}\right)-\notag\\
{}&\lambda^{(0)}\mathcal{O}(\varepsilon^{\frac{m-2}{2}})+\mathcal{O}(e^{-\max(\alpha_1,\alpha_2)}\varepsilon^{m-2})\notag.
\end{align}
Since $||\nabla w_*||_{L^2}=a_*\mathcal{O}(\varepsilon^{m-2})$, we can choose $\alpha_1,\alpha_2$ so that the term $||\nabla w_1||_{L^2}-||\nabla w_2||_{L^2}$ 
dominates the rest of the expression for $\lambda^{(0)}$.
Evidently, one can vary the parameters $a_1$ and $a_2$ so that the sign of 
$\lambda^{(0)}$ -- and hence the sign of 
$\lambda_{\tilde{F}_\varepsilon(\tilde{v}_\varepsilon)}$ -- changes. As we previously noted, 
$\lambda_{\tilde{F}_\varepsilon(\tilde{v}_\varepsilon)}$ depends continuously on $a_1$ and $a_2$,
so we conclude that there are suitable values of $a_1$ and $a_2$
for which the projection term $\lambda_{\tilde{F}_\varepsilon(\tilde{v}_\varepsilon)}$ vanishes.
This finishes the proof of Theorems $1_a, 1_b,$ and $1_c$.
\end{proof}

\section{The non-critical case}
So far, we have produced a family of metrics 
$(1+\tilde{u}_\varepsilon)^{\frac{4}{n-2}}\tilde{g}_\varepsilon$ on $M$, each scalar-flat 
and having constant boundary mean curvature of size $\mathcal{O}(\varepsilon^{m-2})$.
In this section we will prove Theorems $2_a, 2_b,$ and $2_c$, where we arrange for this mean curvature to vanish
entirely. To achieve this, we will need yet another alteration to the above construction.
From now on, we assume that neither of the original manifolds are Ricci-flat with totally 
geodesic boundary , i.e. we assume that $\max(\sup_{M_*}|Ric_{g_*}|,\sup_{\partial M_*}|A_{g_*}|)>0$ 
for both $*=1$ and $2$.  

Let $S_*$ be a positive-definite symmetric 2-tensor with 
\[
\mathrm{spt}(S_*)\subset \left((M_*\setminus \iota^1_*)\cap (\mathrm{spt}(Ric_{g_*})\cup 
	\mathrm{spt}(A_{g_*}))\right).
\]
For a real parameter $\tilde{r}_*$, set $r_*:=\tilde{r}_*\varepsilon^{m-2}$
and consider the following variation of $g_\varepsilon$
\[
\tilde{g}_\varepsilon:=g_\varepsilon+r_1S_1+r_2S_2,\quad \tilde{g}_*:=g_*+r_*S_*,
\]
not to be confused with the conformal modifications made in section 6.

We apply the constructions of sections 4 and 5 to $\tilde{g}_\varepsilon$
in order to produce a family of solutions, $v=v_\varepsilon(r_1,r_2)\in B^\gamma_{r_\varepsilon}$ to
\[
\begin{cases}
\Delta_{\tilde{g}_\varepsilon}v=\tilde{F}_\varepsilon(v,r_1,r_2)&\text{ in } M\\
\partial_\nu v=\tilde{F}^\partial_\varepsilon(v,r_1,r_2)-
	\lambda_{\tilde{F}_\varepsilon(v,r_1,r_2)}\beta_\varepsilon&\text{ on }\partial M
\end{cases}
\]
where 
\[
\tilde{F}_\varepsilon(v,r_1,r_2):=c_nR_{\tilde{g}_\varepsilon}(1+v)
\]
is defined as usual, but
\[
\tilde{F}^\partial_\varepsilon(v,r_1,r_2):=-2c_nH_{\tilde{g}_\varepsilon}(1+v)
\]
has been altered so that, supposing we can arrange for 
$\lambda_{\tilde{F}_\varepsilon(v,r_1,r_2)}=0$, the boundary mean curvature of 
$(1+v)^{\frac{4}{n-2}}\tilde{g}_\varepsilon$ is exactly 0.
As before, we will assume that 
$\int_{\partial M}\beta_\varepsilon d\sigma_{\tilde{g}_\varepsilon}=0$,
which can be achieved for any $r_1$ and $r_2$ by an appropriate 
choice of $\alpha_1$ and $\alpha_2$.

Notice that our choice of $r_*$ ensures 
$R_{\tilde{g}_\varepsilon}$ satisfies the same pointwise bounds 
as in the previous sections.  This will allow us to apply the results of sections 4 and 5 
with trivial modifications once we verify 
\begin{equation}\label{eq:greens}
\int_M \tilde{F}_\varepsilon(v,r_1,r_2)d\mu_{\tilde{g}_\varepsilon}=
\int_{\partial M}\tilde{F}_\varepsilon^\partial(v,r_1,r_2)d\sigma_{\tilde{g}_\varepsilon}.
\end{equation}
The second and final step is to arrange for the vanishing of 
$\lambda_{\tilde{F}_\varepsilon(v,r_1,r_2)}$.

Let us take a moment to explain why simultaneous vanishing of the Ricci tensor and 
second fundamental form can potentially be an obstruction to 
achieving the conclusions of theorem B. Briefly, $(M,g_\varepsilon)$ may be
in the same conformal class as an Einstein metric with Neumann boundary
conditions in the sense of \cite{An} and the total scalar curvature plus mean curvature functional 
$Q(g_\varepsilon)$ may stable under even non-conformal perturbations.
For the metric $\tilde{g}_\varepsilon$, we can follow the calculations of \cite{Ar} to compute
\begin{align}
Q(\tilde{g}_\varepsilon)&=Q(g_\varepsilon)+2c_n\sum_{*=1}^2r_*
\left(\int_Mg_*(S_*,Ric_{g_*})d\mu_{g_*}-\int_{\partial M}g_*(S_*,A_{g_*})d\sigma_{g_*}\right)+\notag\\
{}&\quad\quad \mathcal{O}(r_1^2)+\mathcal{O}(r_2^2)\notag\\
{}&=Q(g_\varepsilon)+2c_n\sum_{*=1}^2r_*
\left(\int_MK_*d\mu_{g_*}-\int_{\partial M}K^\partial_*d\sigma_{g_*}\right)+\mathcal{O}(\varepsilon^{2(m-2)})\notag
\end{align}
where we have introduced the notation $K_*:=g_*(S_*,Ric_{g_*})$ and $K_*^\partial:=g_*(S_*,A_{g_*})$.

From this formula, we can see that if both $Ric_{g_*}$ and $A_{g_*}$ vanish identically for $*=1$ and $2$,
the first variation of $Q(g_\varepsilon)$ vanishes for all choices of $S_*$ and
we will be unable to correct the term $F(g_\varepsilon)$ with a 
small (relative to $\varepsilon$) perturbation of $g_\varepsilon$ away from the 
gluing locus to achieve the desired vanishing mean curvature. 
This reasoning heuristically explains why our construction may fail to 
produce scalar-flat metrics with vanishing boundary mean curvature on $M$
without assumptions on the Ricci tensor and second fundamental form.

\subsection{Achieving the orthogonality condition}
In this subsection, we will give a description of the values
$r_1$ and $r_2$ for which (\ref{eq:greens}) is satisfied.
\begin{prop6}
For small $\varepsilon$ and $v\in B^\gamma_{r_\varepsilon}$, there is a smooth function $f_v$ defined on a neighborhood 
$\overline{U}$ of $\frac{\varepsilon^{m-2}}{2}$ such that
\[
\int_M \tilde{F}_\varepsilon(v,r_1,f_v(r_1))d\mu_{\tilde{g}_\varepsilon}=
\int_{\partial M}\tilde{F}_\varepsilon^\partial(v,r_1,f_v(r_1))d\sigma_{\tilde{g}_\varepsilon}
\]
for all $r_1\in\overline{U}$.
\end{prop6}
\begin{proof}
For any $v\in B^\gamma_{r_\varepsilon}\subset \mathcal{C}^0_\gamma(M)$,
we introduce the function
\begin{align}
G_{v,\varepsilon}(r_1,r_2):=&\frac{1}{c_n}\Big(\int_M \tilde{F}_\varepsilon(v,r_1,r_2)d\mu_{\tilde{g}_\varepsilon}-
	\int_{\partial M}\tilde{F}_\varepsilon^\partial(v,r_1,r_2)d\sigma_{\tilde{g}_\varepsilon}\Big)\notag\\
{}=&\int_MR_{g_\varepsilon}d\mu_{g_\varepsilon}+2\int_{\partial M}H_{g_\varepsilon}d\sigma_{g_\varepsilon}+\sum_{*=1,2}r_*\left(\int_{M_1}K_*d\mu_{g_*}-\int_{\partial M_*}K_*^\partial d\sigma_{g_*}\right)\notag\\
{}&+L_v(r_1,r_2)+Q_v(r_1,r_2)\notag
\end{align}
where we have introduced the notation
\begin{align}
L_v(r_1,r_2):=&\int_MvR_{g_\varepsilon}d\mu_{g_\varepsilon}+\sum_{*=1,2}r_*\left(\int_{M_*}vK_*d\mu_{g_*}-\int_{\partial M_*}vK_*^\partial d\sigma_{g_*}\right)-2\int_{\partial M_*}vH_{g_\varepsilon}d\sigma_{g_\varepsilon}\notag\\
Q_v(r_1,r_2):=&\sum_{*=1,2}\int_{M_*}R_{\tilde{g}_*}(1+v)d\mu_{g_*}-2\int_{\partial M_*}H_{\tilde{g}_\varepsilon}(1+v)\notag\\
{}&-r_*\int_{M_*}K_*(1+v)d\mu_{g_*}+r_*\int_{\partial M_*}K_*^\partial(1+v)d\sigma_{g_*}
	+\mathcal{O}(\varepsilon^{2(m-2)}).\notag
\end{align}
$L_v$ and $Q_v$ can be interpreted as the linear and quadratic parts, 
respectively, of $G_{v,\varepsilon}$.
We also introduce the function $H_\varepsilon(r_1,r_2):=G_{v,\varepsilon}(r_1,r_2)-L_v(r_1,r_2)-Q_v(r_2,r_2)$.

For simplicity, we will pick $S_*$ to satisfying the following conditions. 
We assume that $S_*$ has been chosen so that 
$\int_MK_*d\mu_{g_*}-\int_{\partial M_*}K_*^\partial d\sigma_{g_*}=1$ 
and we will only consider the case when 
\[
\int_{M}R_{g_\varepsilon}d\mu_{g_\varepsilon}+2\int_{\partial M}H_{g_\varepsilon}d\sigma_{g_\varepsilon}<0,
\]
though the argument is very similar if this quantity is positive. 
Since this term is $\mathcal{O}(\varepsilon^{m-2})$, we will scale the metric 
$g_\varepsilon$ so that it is equal to $-\varepsilon^{m-2}$.
Now $H_\varepsilon$ takes the form
\[
H_\varepsilon(r_1,r_2)=-\varepsilon^{m-2}+r_1+r_2
\]
and the vanishing locus of $H_\varepsilon(r_1,r_2)$ is given by $\{(r_1,r_2):r_1+r_2=\varepsilon^{m-2}\}$.
We will see that the zero set of $G_{v,\varepsilon}(r_1,r_2)$ is uniformly close to this set.

It is straight forward to check that there is a constant $C>0$, independent of $\varepsilon$ and 
$v\in B^\gamma_{r_\varepsilon}$, such that
\[
L_v(r_1,r_2),Q_v(r_1,r_2)\leq C_1\varepsilon^{m-2+\gamma}.
\]
So, for any $\eta>0$, there is sufficiently small $\varepsilon$ so that
\[
|L_v(r_1,r_2)|,|Q_v(r_1,r_2)|\leq\frac{\eta}{2}\varepsilon^{m-2}.
\]
It follows that
\begin{align}
\{G_{v,\varepsilon}(r_1,r_2)=0\}&
	=\{(r_1,r_2):r_1+r_2=\varepsilon^{m-2}-L_v(r_1,r_2)-Q_v(r_1,r_2)\}\notag\\
{}&\subset\{(r_1,r_2):(1-\eta)\varepsilon^{m-2}\leq r_1+r_2\leq(1+\eta)\varepsilon^{m-2}\}=:Z_\varepsilon.\notag
\end{align}
From these remarks, we can find many zeroes of $G_{v,\varepsilon}$. For instance, 
setting $r_1':=\varepsilon^{m-2}/2$, for any $v\in B^\gamma_{r_\varepsilon}$, 
there is a number $r_2'=r_2'(v)$
with $(r_1',r_2'(v))\in Z_\varepsilon$ and $G_{v,\varepsilon}(r_1',r_2'(v))=0$.
However, we will still need a degree of freedom to arrange for $\lambda_{\tilde{F}_\varepsilon}=0$
in the next subsection.
Fortunately, for each $v\in B^\gamma_{r_\varepsilon}$ we will 
find a 1-parameter family of solutions near $(r_1',r_2')$ by applying the 
implicit function theorem to $G_{v,\varepsilon}$.

Computing the derivatives of $G_{\varepsilon,v}$,
\begin{align}
\left\lvert\frac{\partial}{\partial r_*}G_{\varepsilon,v}(0,0)\right\rvert&=\left\lvert\int_{M_*}K_*(1+v)d\mu_{g_*}-
	\int_{\partial M_*}K_*^\partial(1+v)d\sigma_{g_*}\right\rvert\notag\\
{}&\geq\left\lvert\int_{M_*}K_*d\mu_{g_*}-\int_{\partial M_*}K_*^\partial d\sigma_{g_*}\right\rvert\notag\\
{}&\quad\quad\quad-||v||_{\mathcal{C}^0(M)}\left(\int_{M_*}|K_*|d\mu_{g_*}+\int_{\partial M_*}|K_*^\partial|d\sigma_{g_*}\right)\notag\\
{}&\geq\frac12\notag
\end{align}
for $*=1,2$ and all $v\in B^\gamma_{r_\varepsilon}$.  From this
we can find a radius $R>0$, uniform in $\varepsilon$ and $v\in B^\gamma_{r_\varepsilon}$, 
so that that $\left\lvert\frac{\partial}{\partial r_*}G_{v,\varepsilon}\right\rvert\geq \frac{1}{4}$ 
on $B_R(0)\subset\mathbb{R}^2$.

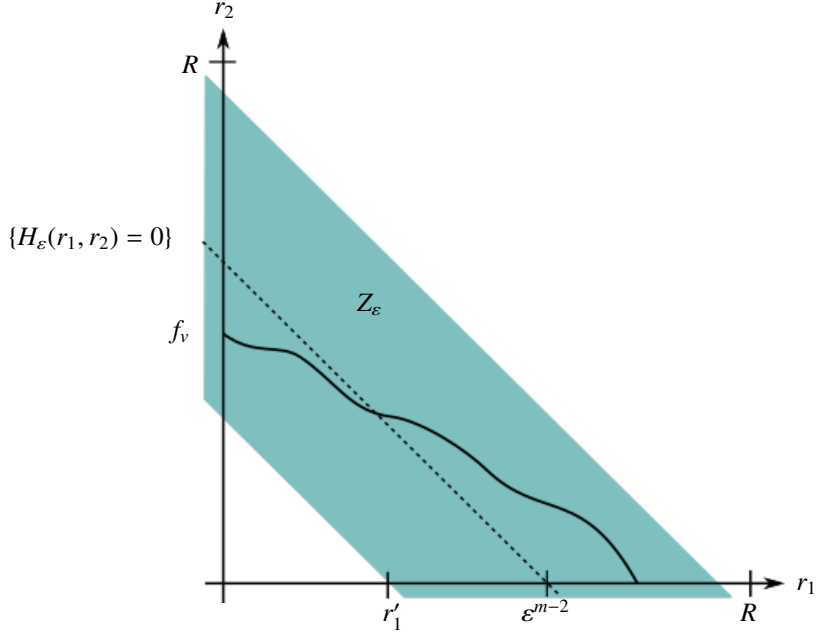
\begin{figure}[htb!]
\begin{center}
\begin{picture}(0,0)
\put(225,5){$r_1$}
\put(7,222){$r_2$}
\put(-5,200){$R$}
\put(204,-7){$R$}
\put(-70,135){$\{H_\varepsilon(r_1,r_2)=0\}$}
\put(-10,100){$f_v$}
\put(70,-7){$r_1'$}
\put(122,-7){$\varepsilon^{m-2}$}
\end{picture}
\includegraphics[height=3in]{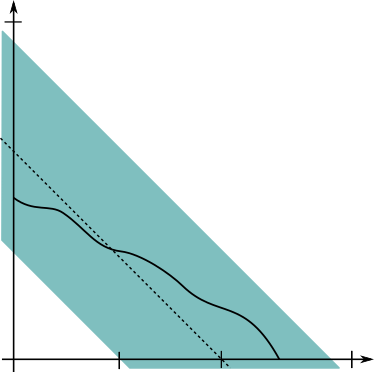}
\put(-160,110){$Z_\varepsilon$}
\end{center}
\caption{The region $Z_\varepsilon$ (in blue) and the function $f_v$ in the $r_1r_2$-plane.}
\label{fig:Zepsilon}
\end{figure}

After perhaps restricting to smaller $\varepsilon$, the set $Z_\varepsilon\cap\{r_1,r_2\geq0\}$
is contained in $B_R(0)$. We may now apply the implicit function theorem on $G_{v,\varepsilon}$
about the points $(r_1',r_2'(v))$ to obtain, for every $v\in B^\gamma_{r_\varepsilon}$, open 
neighborhoods $U(v)$ and $V(v)$ containing $r_1'$ and $r_2'(v)$, respectively, and a function
$f_v:U(v)\to V(v)$ so that $G_{v,\varepsilon}(r_1,f_v(r_1))=0$ for all $r\in U(v)$ 
(see Figure \ref{fig:Zepsilon}).
In fact, we know apriori that $f_v$ can be extended to the interval 
$(0,(1-\eta)\varepsilon^{m-2})$,
and so we may choose open sets $U$ and $V$ which are independent of 
$v\in B^\gamma_{r_\varepsilon}$.
Since the graph of each $f_v$ lies in $Z_\varepsilon$ they may 
be extended to $f_v:\overline{U}\to\overline{V}$.
\end{proof}

Before we continue, we will need one more property of 
the family $\{f_v\}_{v\in B^\gamma_{r_\varepsilon}}$.  By construction, we have
\[
f_v(r_1)=\int_{M}R_{g_\varepsilon}d\mu_{g_\varepsilon}+
2\int_{\partial M}H_{g_\varepsilon}(1+v)d\sigma_{g_\varepsilon}-r_1+
L_v(r_1,f_v(r_1))+Q_v(r_1,r_2).
\]
From this one can see, for small $\varepsilon$ and any $r_1,r_1'\in\overline{U}$, that
\[
|f_v(r_1)-f_v(r_1')|\leq4|r_1-r_1'|.
\]
Now Ascoli-Arzela tells us that
$\{f_v\}_{v\in B^\gamma_{r_\varepsilon}}$ is precompact in the 
$\mathcal{C}^0(\overline{U})$ norm.
This function $f$ will have the same Lipschitz norm bound.

\subsection{Vanishing of the rough projection}
Paralleling section 4, we introduce the map
$\tilde{P}_{\varepsilon}:\mathcal{C}^0_\gamma(M)\to\mathcal{C}^0_\gamma(M)$
sending a function $v$ to the solution of 
\[
\begin{cases}
\Delta_{\tilde{g}_\varepsilon}\tilde{P}_{\varepsilon}(v)=\tilde{F}_\varepsilon(v,r'_1,f_v(r'_1))&\text{ in }M\\
\partial_\nu\tilde{P}_{\varepsilon}(v)=
\tilde{F}^\partial_\varepsilon(v,r'_1,f_v(r'_1))-
\lambda_{\tilde{F}_\varepsilon(v,r'_1,f_v(r'_1))}\beta_\varepsilon&\text{ on }\partial M
\end{cases}
\]
The arguments of that section can be repeated to
show $\tilde{P}_{\varepsilon}$ is also a contraction mapping on $B^\gamma_{r_\varepsilon}$
for small $\varepsilon$ and $\gamma\in(0,\frac14)$.
This shows that 
$\{(\tilde{P}_{\varepsilon})^j(0)\}_{j=1}^\infty$ converges to a fixed point 
$\tilde{v}_{\varepsilon}\in B^\gamma_{r_\varepsilon}$ with respect to the $\mathcal{C}^0_\gamma$-norm. 
From the previous section,
after passing to a subsequence, the functions $f_{(\tilde{P}_{\varepsilon})^j(0)}$ also
converge to a continuous function $f:\overline{U}\to \overline{V}$ which verifies the orthogonality
condition for $\tilde{v}_\varepsilon$. We conclude that, for any $r_1\in\overline{U}$, we have
\begin{equation}\label{eq:preyamabe2}
\begin{cases}
\Delta_{\tilde{g}_\varepsilon}\tilde{v}_\varepsilon=\tilde{F}_\varepsilon(\tilde{v}_\varepsilon,r_1,f(r_1))&\text{ in }M\\
\partial_\nu\tilde{v}_\varepsilon=\tilde{F}^\partial_\varepsilon(\tilde{v}_\varepsilon,r_1,f(r_1))-
	\lambda_{\tilde{F}_\varepsilon(\tilde{v}_\varepsilon,r_1,f(r_1))}\beta_\varepsilon&\text{ on }\partial M.
\end{cases}
\end{equation}
The following proposition will complete the proof of Theorems $2_a,2_b,$ and $2_c$.

\begin{prop7}
There exists an $\varepsilon_0>0$ so that for all $\varepsilon\in(0,\varepsilon_0)$
there is a choice of $r_1\in\overline{U}$ for which
$\lambda_{\tilde{F}_\varepsilon(\tilde{v}_\varepsilon,r_1,f(r_1))}$ vanishes where $\tilde{v}_\varepsilon$ is given by (\ref{eq:preyamabe2}).
\end{prop7}

\begin{proof}
Since $\lambda_{\tilde{F}_\varepsilon(\tilde{v}_\varepsilon,r_1,f(r_1))}$ 
is continuous in $r_1$, it suffices to show that its sign can be controlled by $r_1\in\overline{U}$.
Following section 5, for small $\varepsilon$, the sign of 
$\lambda_{\tilde{F}_\varepsilon(\tilde{v}_\varepsilon,r_1,f(r_1))}$ is controlled by the sign of
\begin{align}
\lambda^{(0)}&
	=\frac{1}{\int_{\partial M}(\rho_1+\rho_2)d\sigma_{\tilde{g}_\varepsilon}}
	\Big(\int_M\tilde{F}_\varepsilon(\tilde{v}_\varepsilon,r_1,f(r_1))\beta_\varepsilon 
	d\mu_{\tilde{g}_\varepsilon}\notag\\
{}&\quad-\int_{\partial M}\tilde{F}^\partial_\varepsilon(\tilde{v}_\varepsilon,r_1,f(r_1))\beta_\varepsilon 
	d\sigma_{\tilde{g}_\varepsilon}+\int_M(\Delta_{\tilde{g}_\varepsilon}(\rho_1\tilde{u}_T)\notag\\
{}&\quad-\Delta_{\tilde{g}_\varepsilon}(\rho_2\tilde{u}_T))d_{\tilde{g}_\varepsilon}-\int_{\partial M}\partial_\nu(\rho_1\tilde{u}_T)-\partial_\nu(\rho_2\tilde{U}_T)d\sigma_{\tilde{g}_\varepsilon}\Big).\notag
\end{align}
As before, we have
\[
\int_M\Delta_{\tilde{g}_\varepsilon}(\rho_*\tilde{u}_p^\varepsilon)d\mu_{\tilde{g}_\varepsilon}+\int_{\partial M}\partial_\nu(\rho_*\tilde{u}_T)d\sigma_{\tilde{g}_\varepsilon}=
	\mathcal{O}(e^{-\alpha_*}\varepsilon^{m-2})
\]
for $*=1,2$. For the first term appearing in the above expression for $\lambda^{(0)}$, we have
\begin{align}
\frac{1}{c_n}\int_M\tilde{F}_\varepsilon(\tilde{v}_\varepsilon,r_1,f(r_1))d\mu_{\tilde{g}_\varepsilon}&=
	r_1\int_{M_1}K_1d\mu_{g_1}-f(r_1)\int_{M_2}K_2d\mu_{g_2}\notag\\
{}&+\int_MR_{g_\varepsilon}\beta_\varepsilon dvol_{g_\varepsilon}+
	\mathcal{O}(\varepsilon^{m-2+2\gamma})\notag\\
{}&=r_1\int_{M_1}K_1d\mu_{g_1}-f(r_1)\int_{M_2}K_2d\mu_{g_2}+\mathcal{O}(e^{-\min(\alpha_1,\alpha_2)}\varepsilon^{m-2})\notag
\end{align}
The boundary term has a similar estimate
\begin{align}
\frac{1}{c_n}\int_{\partial M}\tilde{F}_\varepsilon^\partial(\tilde{v}_\varepsilon,r_1,f(r_1))\beta_\varepsilon d\sigma_{\tilde{g}_\varepsilon}=&-r_1\int_{\partial M_1}K^\partial_1d\sigma_{g_1}+f(r_1)\int_{\partial M_2}K_2^\partial d\sigma_{g_2}\notag\\
{}&\quad+\mathcal{O}(e^{-\min(\alpha_1,\alpha_2)}\varepsilon^{m-2})\notag.
\end{align}
Summing these three expressions together gives us the expression we are looking for
\[
\lambda^{(0)}=r_1-f(r_1)+\mathcal{O}(e^{-\max(\alpha_1,\alpha_2)}\varepsilon^{m-2}).
\]
Hence, we can choose large $\alpha_1$ and $\alpha_2$ so that
the sign of 
$\lambda_{\tilde{F}_\varepsilon(\tilde{v}_\varepsilon,r_1,f(r_1))}$ is controlled by $r_1-f(r_1)$.  
Evidently, the graph of $f$ must intersect the line $\{r_1=r_2\}$ in $Z_\varepsilon$ (see Figure \ref{fig:Zepsilon}) and we 
conclude that the sign of $r_1-f(r_1)$ changes as $r_1$ varies over $\overline{U}$, finishing the proof of Proposition 7.
\end{proof}

\section*{References}
\bibliographystyle{elsarticle-harv}
\bibliography{references}
\end{document}